\documentclass[11pt,reqno]{amsart}
\usepackage[colorlinks]{hyperref}
\hypersetup{colorlinks=true,linkcolor={magenta},pdfborder=0 0 0}
\usepackage{latexsym}
\usepackage{amsmath}
\usepackage{enumerate}
\usepackage{amsfonts}
\usepackage{amssymb}
\usepackage{latexsym}
\usepackage{enumitem}
\usepackage{comment}
\usepackage{fixmath}
\usepackage[usenames,dvipsnames]{color}
\usepackage{mathdots}
\usepackage{multicol}
\usepackage{cleveref}
\usepackage{graphicx}
\usepackage{subcaption} 
\usepackage[boxed, linesnumbered, noend, noline]{algorithm2e}
\usepackage{comment}
\usepackage[utf8]{inputenc}
\usepackage[english]{babel}
\usepackage[T1]{fontenc}
\usepackage{fourier}
\usepackage{bbm}
\usepackage{color}
\usepackage{fullpage}
\usepackage{enumitem}
\usepackage{pgfplots}
\pgfplotsset{compat=newest}
\usepackage{graphicx}
\usepackage{caption}
\usepackage{float} 

\usepackage{enumitem}
\usepackage{yfonts}
\usepackage{eufrak}
\usepackage{mathrsfs}
\usepackage
{xcolor}

\numberwithin{equation}{section}

\renewcommand{\vec}[1]{\boldsymbol{#1}}

\newcommand{\PP}{\mathbb{P}}

\newcommand\vD{\vec D}

\newcommand\vG{\vec G}

\newcommand\vx{\vec x}

\newcommand\eps{\varepsilon}

\newcommand\Erw{\mathbb{E}}
\newcommand\ex{\Erw}

\newcommand{\Po}{{\rm Po}}

\newcommand\bc[1]{\left({#1}\right)}
\newcommand\cbc[1]{\left\{{#1}\right\}}

\newcommand\abs[1]{\left|{#1}\right|}

\newcommand{\Erdos}{Erd\"os}
\newcommand{\Renyi}{R\'enyi}

\newcommand{\Bollobas}{Bollob\'as}

\newcommand\pr{\mathbb{P}} 
\renewcommand\Pr{\pr} 
\newcommand\Lem{Lemma}
\newcommand\Prop{Proposition}
\newcommand\Thm{Theorem}
\newcommand\Def{Definition}

\newcommand\Sec{Section}

\newcommand\Rem{Remark}

\usepackage[hang,flushmargin]{footmisc} 

\renewcommand{\footnoterule}{%
	\kern -3pt
	\hrule width \textwidth height 0.4pt
	\kern 2.6pt
}

\newtheorem{definition}{Definition}[section]
\newtheorem{claim}[definition]{Claim}

\newtheorem{remark}[definition]{Remark}
\newtheorem{theorem}[definition]{Theorem}
\newtheorem{condition}[definition]{Condition}
\newtheorem{lemma}[definition]{Lemma}
\newtheorem{proposition}[definition]{Proposition}

\newtheorem{fact}[definition]{Fact}

\newcommand{\floor}[1]{\left\lfloor#1\right\rfloor}

\newcommand{\ind}[1]{{\pmb 1}\{#1\}}

\SetKwInput{KwData}{Input}
\SetKwInput{KwResult}{Output}

\def\pr{{\mathbb P}}

\def\ks{\mathcal{K}_{\textbf{CM}}}
\def\ph{{\widehat{\phi}}}

\renewcommand{\aa}{\alpha_\ast}
\newcommand{\ab}{\alpha^{\ast}}

\newcommand{\hphi}{\widehat{\phi}}
\newcommand{\CM}{\textbf{CM}}

\newcommand{\msgL}{\textcolor{red}{\textbf{L}}}
\newcommand{\msgM}{\textcolor{blue}{\textbf{M}}}
\newcommand{\msgU}{\textcolor{green}{\textbf{U}}}

\newcommand{\msg}[3]{\mu_{#1\to #2}^{(#3)}}

\newcommand{\kll}{\star}

\newcommand\biasp{\widehat{p}}

\newcommand{\dtv}{d_{\mathrm{TV}}}

\usepackage{tikz}
\usetikzlibrary{positioning, decorations.markings}
\usetikzlibrary{calc}

\begin{document}
	\title{Asymptotic size of the Karp-Sipser Core in Configuration Model}
	\author{
		Arnab Chatterjee$^{*}$}
	\thanks{\raggedright${*}$TU Dortmund, Faculty of Computer Science, Otto-Hahn-Str. 12, 44227 Dortmund, Germany. \texttt{arnab.chatterjee@tu-dortmund.de}} 
	\author{Joon Hyung Lee$^{\dagger}$}
	\thanks{\raggedright${\dagger}$Leiden Institute of Advanced Computer Science, Einsteinweg 55, 2333 CC Leiden, the Netherlands. \texttt{j.h.lee@liacs.leidenuniv.nl}} 
	\author{Haodong Zhu$^{\ddagger}$}
	\thanks{\raggedright${\ddagger}$Eindhoven University of Technology, Department of Mathematics and Computer Science, 5600 MB Eindhoven, the Netherlands. \texttt{h.zhu1@tue.nl}} 
	
	\maketitle
	
	\begin{abstract}
		We study the asymptotic size of the Karp-Sipser core in the configuration model with arbitrary degree distributions. 
		The Karp–Sipser core is the induced subgraph obtained by iteratively removing all leaves and their neighbors through the leaf-removal process, and finally discarding any isolated vertices \cite{BCC}.
		Our main result establishes the convergence of the Karp-Sipser core size to an explicit fixed-point equation under general degree assumptions.
		The approach is based on analyzing the corresponding local weak limit of the configuration model -- a unimodular Galton-Watson tree and tracing the evolution process of all vertex states under leaf-removal dynamics by use of the working mechanism of an enhanced version of Warning Propagation along with Node Labeling Propagation.
		
		\hfill MSc: 05C80, 60C05, 68W20    
	\end{abstract}
	
	\section{Introduction and Results}\label{sec_intro}
	\subsection{Background and motivation}\label{sec_background}
	Matchings have a long and exciting history in combinatorics and computer science.
	In the words of Lov\'asz and Plummer \cite{LovaszPlummer} "Matching theory has often been in attendance when many of the exciting new concepts in combinatorial optimization have been born."
	Going back to the theoretical results from the 1960s, one of the fundamental theorems in random graph theory, established by \Erdos \, and \Renyi \cite{ER1} describes the asymptotic probability that a random graph possesses a \emph{perfect matching}. 
	A matching on a finite graph $\mathcal{G}=(V,E)$ is a subset of pairwise non-adjacent edges $E'\subseteq E$. Then the total number of isolated vertices $|V|-2|E'|$ of $(V,E')$ is said to be revealed by $E'$.
	Obviously, a perfect matching can exist only if the number of vertices is even. 
	In 1947, Tutte characterized all of the graphs with a perfect matching \cite{LovaszPlummer} by stating that a graph $\mathcal{G}$ has a perfect matching iff there is no set of vertices $V'$ with the property that $\mathcal{G}\setminus V'$ contains more than $|V'|$ components with odd number of vertices.
	For a given graph, define a leaf-component to be a connected component of a subgraph obtained by a single edge deletion from the graph, which again has the property that any further deletion of one edge wouldn't disconnect it.
	Further, in 1965 Edmonds \cite{LovaszPlummer} established the first polynomial time algorithm for obtaining a maximum matching.
	
	Besides this, matching in the \Erdos-\Renyi \, random graph $\mathcal{G}(n,M)$ has been extensively studied in the literature. 
	In the paper \cite{ER1} \Erdos \, and \Renyi~ proved the threshold for the existence of a perfect matching in the ER-random graph.
	More precisely, let $n$ be even and a sequence of real numbers $c_1,c_2,\cdots,c_n$, then for the \Erdos-\Renyi \, random graph with $n$ vertices and $M=\floor{\frac{n(\log n+c_n)}{2}}$ edges, as $n\to \infty$ we have
	\begin{align*}
		\Pr\left(\mathcal{G}\left(n,M=\floor{\frac{n(\log n+c_n)}{2}}\right)\mbox{has a perfect matching}\right)
		\to
		\begin{cases}
			0  & \text{if} \: c_n \to -\infty, \\ 
			e^{-e^{-c}} & \text{if} \: c_n \to c, \\ 
			1 & \text{if} \: c_n \to +\infty.
		\end{cases}
	\end{align*} 
	The groundbreaking work of obtaining the near maximal matching on a given graph was first introduced in 1981 by Karp and Sipser in their seminal paper \cite{KarpSipser} known as \emph{Karp-Sipser leaf removal algorithm}. 
	The initial motivation of Karp and Sipser to consider this algorithm was that the leaves and isolated vertices removed during their process form an independent set of $\vG$ with very high density. 
	Although the finding of the maximal size independent set problem is an NP-hard problem, the Karp-Sipser algorithm provides a fair lower bound on this.
	Further, they showed in their paper \cite{KarpSipser} that the algorithm is more efficient in a sparse \Erdos-\Renyi \, random graph. 
	Later, Aronson, Frieze and Pittel \cite{KSR} demonstrated that this algorithm produces a maximum matching with high probability when the average degree is less than $e$, and provides a matching that is sub-linearly close to the maximum size when the average degree is greater than $e$.  
	\begin{algorithm}[h!]
		\KwData{ An undirected graph $(\mathcal{G} = V, E)$} 
		\KwResult{ A matching $ E' \subseteq E$}  
		
		\begin{enumerate}
			\item Initialize \( E' \gets \emptyset \),  $V \gets V\backslash\cbc{v:~\deg(v) = 0}$.
			\item \label{it_2}\textbf{While} $\exists$ a vertex $( v \in V )$ such that $\deg(v) = 1 $ \textbf{do}:  
			\begin{enumerate}
				\item Let \( u \) be the unique neighborhood of $v$.
				\item Add edge $\{u, v\}$ to the matching: $E'\gets E' \cup \{\{u, v\}\} $.
				\item Remove $u$ and $v$ from the graph:
				$V \gets V \setminus (u, v)$ and 
				$E \gets E \setminus \{x, y\} \in E : x \in (u,v) \text{ or }$ \\ 
				$y \in (u, v)$.
				\item  $V \gets V\backslash\cbc{v:~\deg(v) = 0}$.
			\end{enumerate}
			\item \textbf{If} \( E = \emptyset \), \textbf{terminate}.
			\item \textbf{While} \( E \neq \emptyset \) \label{it_4}\textbf{do}:
			\begin{enumerate}
				\item Choose an arbitrary edge  $\{x, y\} \in E$.
				\item Add $\{x, y\}$ to the matching: $E' \gets E' \cup \{\{x, y\}\}$.
				\item Remove \( x \) and \( y \) from the graph:
				$V \gets V \setminus (x,y)$ and 
				$E \gets E \setminus \{s, t\} \in E : s \in (x, y) \text{ or } t \in (x, y)$
				\item Return to Step 1.
			\end{enumerate}
			\item \textbf{Return} \( E' \).
		\end{enumerate}
		\caption{The Karp-Sipser Algorithm}\label{alg_karpsipser}
	\end{algorithm}
	
	Surprisingly, the algorithm~\ref{alg_karpsipser} does not guarantee an optimal matching. 
	Consider the graph in Figure~\ref{fig:karp-sipser-example} (the complete graph on four vertices with the edge $\{1,2\}$ removed). 
	If in step~(\ref{it_4}) the algorithm chooses $\{x,y\}=\{3,4\}$, then all edges in the graph are deleted and the result is $E'=\{\{3,4\}\}$. 
	However, the matching $\{\{1,3\},\{2,4\}\}$ has size~2 and is therefore strictly better.
	\begin{figure}[htbp]
		\centering
		\begin{tikzpicture}[
			every node/.style={circle,draw,minimum size=18pt,inner sep=0pt},
			node distance=2.2cm
			]
			\node (1) {1};
			\node[right=of 1] (2) {2};
			\node[above=of 1] (3) {3};
			\node[above=of 2] (4) {4};
			
			\draw[thick] (1)--(3);
			\draw[thick] (1)--(4);
			\draw[thick] (2)--(3);
			\draw[thick] (2)--(4);
			\draw[thick] (3)--(4);
		\end{tikzpicture}
		\caption{An example input for the Algorithm~\ref{alg_karpsipser}.}
		\label{fig:karp-sipser-example}
	\end{figure}
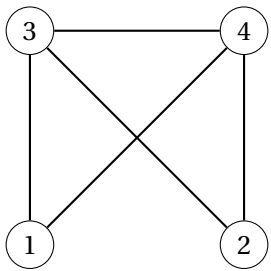
	
	The \emph{Karp–Sipser core}, which arises after the Karp–Sipser algorithm (pseudocode in Algorithm~\ref{alg_karpsipser}) step (\ref{it_2}), is the graph obtained by repeatedly removing all degree-$1$ vertices together with their neighbors, excluding isolated vertices.
	Like in the $k$-core of a random graph, Bauer and Golinelli showed \cite{BGC} that the Karp-Sipser core is independent of how we choose the vertices to remove in the first iteration, which indicates that the Karp-Sipser core is an intrinsic characteristic of the graph.
	The study of Karp-Sipser core not only assists us in finding a near maximum matching but also has applications in other areas, such as independent set \cite{GNS} and the rank of the adjacency matrices \cite{BLS} -- thus helps us gain a better understanding of large scale networks.
	However, most of the papers only discuss the Karp-Sipser core in the sparse \Erdos-\Renyi \, random graph, and there are very few results in the general random graph models.
	The main reason of picking up the sparse \Erdos-\Renyi \, random graph is that the Karp-Sipser core has an almost perfect matching \cite{BLS,BLS2} in that model and the same property applies to the simple configuration model investigated in \cite{BCC}.
	
	The paper aims to study the size of the Karp-Sipser core in the general configuration model and estimate its asymptotic behavior.  
	A first step towards the Karp-Sipser core in the configuration model was undertaken in a recent contribution by Budzinski and Alice \cite{BC}.
	However, their approach relies on the differential equation method and builds on their previous work on a configuration model with bounded degrees.
	Instead, here we establish a limiting distribution for the Karp-Sipser core in the general configuration model assuming the bounded moment conditions on the degree distribution and the presence of degree-$1$ vertices.
	Specifically, we examine the corresponding Galton-Watson tree, which serves as a limiting object of the "local structure" of the configuration model and can be formalized nicely in the language of \emph{"local weak convergence"} \cite{AS,BordenaveCaputo,BS}.
	Our main result uses the working mechanism of the enhanced version of the Warning Propagation algorithm along with Node Labeling and characterizes it in terms of fixed-point recursive equations involving the degree generating function and its size-biased counterpart.
	We rigorously prove the matching bound on the size of the Karp-Sipser core \emph{with high probability} ('w.h.p.') which confirms the tightness of our asymptotic characterization. 
	This is the first result on the asymptotic size of the Karp-Sipser core in any general configuration model with arbitrary degree distributions.
	\subsection{The Configuration model}
	In the context of regular graphs, \emph{configuration model} was first introduced in 1980 by \Bollobas \cite{Bollobas} which was again drawn to attention after the work by Molloy and Reed \cite{MolloyReed} in 1995.
	This model is a widely used method for constructing random graphs with a prescribed degree distribution.
	It allows us to generate random graphs with arbitrary degree sequences, including heavy-tailed distributions, and is suited particularly for the asymptotic analysis in the sparse regime.
	\begin{definition}\label{ConfigModel}
		Let $n\in \mathbb{N}$ be the number of vertices in the random graph. Consider a sequence of degrees $d^{(n)}=(d_{1}^{(n)},d_{2}^{(n)},\cdots,d_{n}^{(n)}) \in \mathbb{N}^{n}$, where each $d_{i}^{(n)}$ denotes the degree assigned to the vertex $i$ of our graph.
		The configuration model denoted by, $\CM_n(d^{(n)})$ is defined as follows:
		\begin{enumerate}[label=(\roman*).]
			\item A set of $d_{i}^{(n)}$ half edges is assigned to each vertex $i$.
			\item Pair the $\sum_{i=1}^{n} d_{i}^{(n)}$ half-edges uniformly at random to form full edges.
		\end{enumerate}
	\end{definition}
	As we aim to construct an undirected (multi) graph with $n$ vertices, where each vertex $j$ has degree $d_{j}$, we assume that the total degree is even i.e. $\sum_{i=1}^{n} d_{i}^{(n)} \equiv 0\mod 2$.
	The resultant random graph may contain self-loops and multiple edges, but for the sake of simplicity, we consider a simple multigraph (although it is not always possible to construct such a simple graph with a given degree sequence) which does not contain any self-loops and multiple edges between any pair of vertices.
	One possible way to construct such a simple multi-graph with a prescribed degree sequence is to pair the half-edges uniformly attached to the different vertices of the graph, and thus form an edge by adding two half-edges.
	We continue this mechanism of randomly choosing and pairing the half edges until all half edges are connected and call the resulting graph the \textit{configuration model} with degree sequence $d^{(n)}$, abbreviated as $\CM_n(d^{(n)})$.
	
	Throughout the paper, we denote our degree sequence by $d^{(n)}$. 
	In particular, for each fixed $n$, since we have a degree sequence of our configuration model, we are dealing with finite-degree sequences.
	So, instead of writing the degree sequence as $d^{(n)}=(d_{1}^{(n)},d_{2}^{(n)},\cdots,d_{n}^{(n)})$, we keep the notation simple.  
	
	We study the asymptotic behavior of the model $\CM_n$ under the assumption that the empirical degree distribution converges to a fixed distribution as $n \to \infty$.
	Let $\vD_n$ be a random variable chosen uniformly from $(d_{1}^{(n)},d_{2}^{(n)},\cdots,d_{n}^{(n)})$ with $d_{i}=d_{i}^{(n)}$ being the degree of vertex $i$ in $\CM_n$.  Also, let $\vD$ be a random variable such that $\pr[\vD=k]=p_{k}$. 
	Then the empirical distribution of the vertex degrees is given by
	\begin{align*}
		\varsigma_{n}(z)=\frac{1}{n}\sum_{i=1}^{n}\mathbb{1}_{\{d_{i}\leq z\}}
	\end{align*}
	where $\varsigma_{n}$ is the distribution function of $\vD_{n}$. Now, we impose the regularity conditions \cite{Remco} for the degrees of the vertex assuming that the degrees of the vertex satisfy the following conditions.
	
	\begin{condition}[Regularity conditions for vertex degrees.]\label{CM_cond}
		\leavevmode\\[-2ex]
		\begin{enumerate}[label=(\roman*).]
			\item \textbf{Weak Convergence of vertex degrees:} There exists a distribution function $\varsigma$ such that for any $z\in \mathbb{R}$,
			\begin{align}\label{CM_wc}
				\lim_{n\to\infty}\pr[\vD_n\leq z]=\pr[\vD\leq z]
			\end{align}
			where $\vD_{n}$ and $\vD$ have the distribution functions $\varsigma_{n}$ and $\varsigma$, respectively.
			Equivalently, we can say that $\vD_n \xrightarrow{d}\vD$.
			\item \textbf{Convergence of the average vertex degrees:} We assume that $\varsigma(0)=0$, i.e., $\pr[\vD\geq 1]=1$ as well as the finite first moment on the convergence of the average vertex degrees.
			\begin{align}\label{CM_first_moment}
				\lim_{n\to \infty} \ex[\vD_{n}]=\ex[\vD] < \infty
			\end{align}
			\item \textbf{Bounded vertex degrees:} We assume the vertex degree have a uniform upper bound for all $n$.
		\end{enumerate}
	\end{condition}
	From the above condition~\ref{CM_cond} since we know that the degrees $d_i$ only take integer values, so does $\vD_n$; therefore, the limiting random variable must be $\vD$. 
	As a consequence, the limiting distribution function $\varsigma$ is constant between integers which implies $\varsigma$ does have discontinuity points, and the weak convergence usually implies \ref{CM_wc} only at continuity points.

	\subsection{Local Weak Limit of the Configuration Model}
	To study the local structures of the configuration model, particularly the asymptotic behavior of the Karp-Sipser algorithm near a typical vertex, it is convenient to use the theory of the local weak convergence of random graphs.
	The limiting objects of our configuration model $\CM_n$ is an infinite rooted tree which captures the local neighborhood of a vertex chosen uniformly at random as $n\to \infty$.
	
	Let $\mathcal{G^{*}}$ be the space of connected, rooted, locally finite graphs equipped with the local topology, more specifically, it converges upto a finite radius. 
	We also define the random rooted graph $(\CM_n,\rho_n)$, where $
	\rho_n$ is a vertex chosen uniformly at random.  
	Moreover, the sequence $(\CM_n,\rho_n)$ converges in distribution in $\mathcal{G^{*}}$ to a Galton-Watson tree $\mathcal{T}$ with offspring distribution of rooted vertex is $D$ and each non-root vertex in $\mathcal{T}$ has offspring distribution is $\widehat{D}$. Then for any $k \geq 1$ we have
	\begin{align}\label{edge_biased_dist}
		\biasp_{k} := \pr[\widehat{D}=k-1]=\frac{kp_{k}}{\sum_{j\ge 0} jp_{j}}
	\end{align}
	Equivalently, if $\phi(\alpha)=\sum_{k\geq 0} p_{k}\alpha^{k}$ be the generating function of $D$, then $\widehat{D}$ (offspring distribution of the vertices other than root in $\mathcal{T}$) has the generating function
	\begin{align}\label{nr_gen_func}
		\widehat{\phi}(\alpha)&=\frac{\phi'(\alpha)}{\phi'(1)}=\frac{\sum_{k\geq 1}k p_k \alpha^{k-1}}{\sum_{k\geq 1} k p_k} 
	\end{align}
	It is known that when $\sum_{k\geq 0} k(k-2)p_k\leq 0$, the size of the Karp-Sipser core is sublinear.
	Hence, we assume in the rest of the paper that $\sum_{k\geq 0} k(k-2)p_k > 0$ as well as $p_1 >0$.  This condition corresponds to the emergence of a giant core, in analogy with the giant component threshold first studied in random graphs with given degrees by Molloy and Reed~\cite{MolloyReed}.
	Thus, the Galton-Watson tree $\mathcal{T}$ serves in the asymptotic regime as the canonical object for analyzing local algorithms including the Karp-Sipser leaf removal process. 
	In the next section, we will introduce the Karp-Sipser leaf removal algorithm both in graphs as well as in a rooted Galton-Watson tree which serves as a local limit of our configuration model $\CM_n$ respectively. 
	The obvious question is why we treat the Karp-Sipser leaf removal algorithm differently in the above two cases. 
	This is because in the rooted Galton-Watson tree, the offspring distribution of the root and that of other vertices are different, which is not the case in random graphs. 
	
	\subsection{The main result} To state the main result, we need to consider the function 
	\begin{align}\label{zeta_func}
		\zeta(\alpha)=\alpha+\ph(1-\ph(\alpha))-1
	\end{align}  
	Let $\alpha_{*}\in [0,1]$ be the smallest and $\alpha^{*}\geq \alpha_* \in [0,1]$ the largest solution of $\zeta(\alpha)=0$.
	The following theorem provides an asymptotic size of the Karp-Sipser core in the general configuration model, given conditions listed in the statement of the theorem.
	\begin{figure}[H]
		\centering
		\begin{subfigure}[]{0.33\textwidth}
			\centering
			\includegraphics[width=\textwidth]{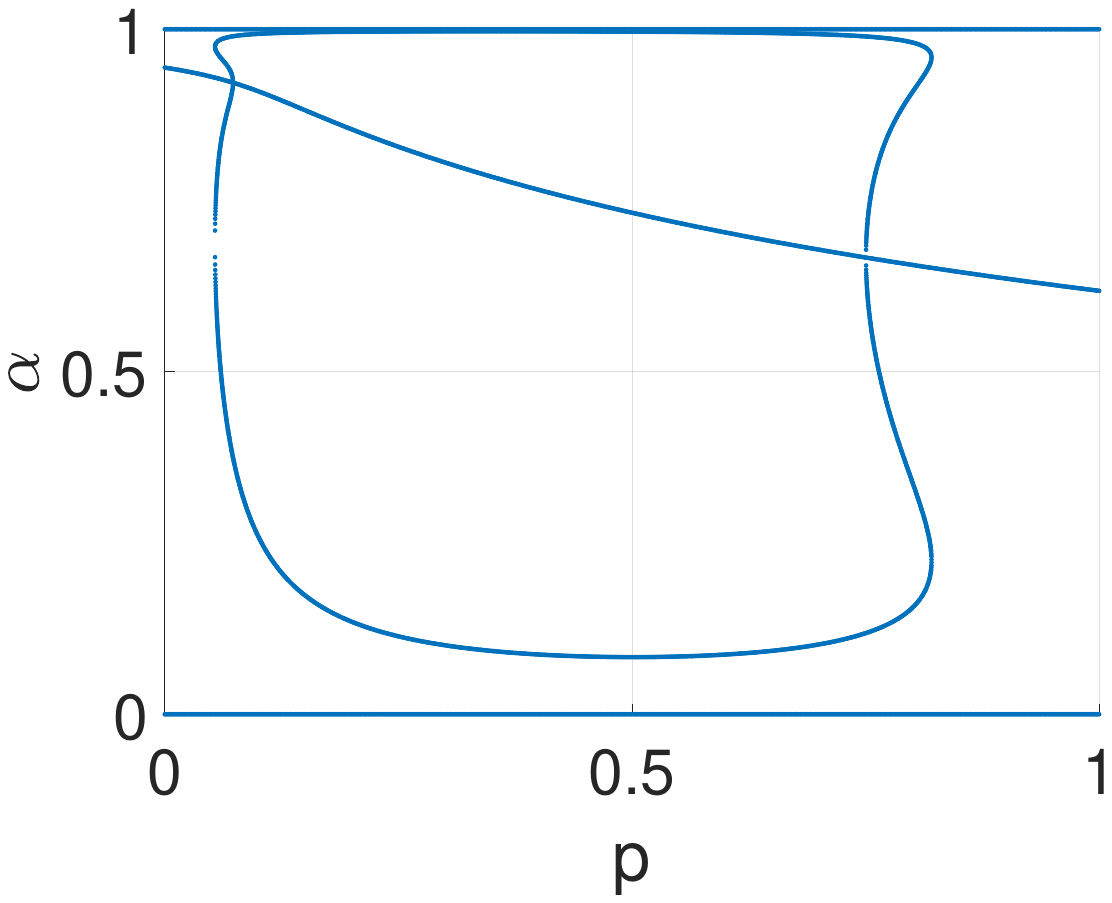}
		\end{subfigure}
		\hfill
		\begin{subfigure}[]{0.33\textwidth}
			\centering
			\includegraphics[width=\textwidth]{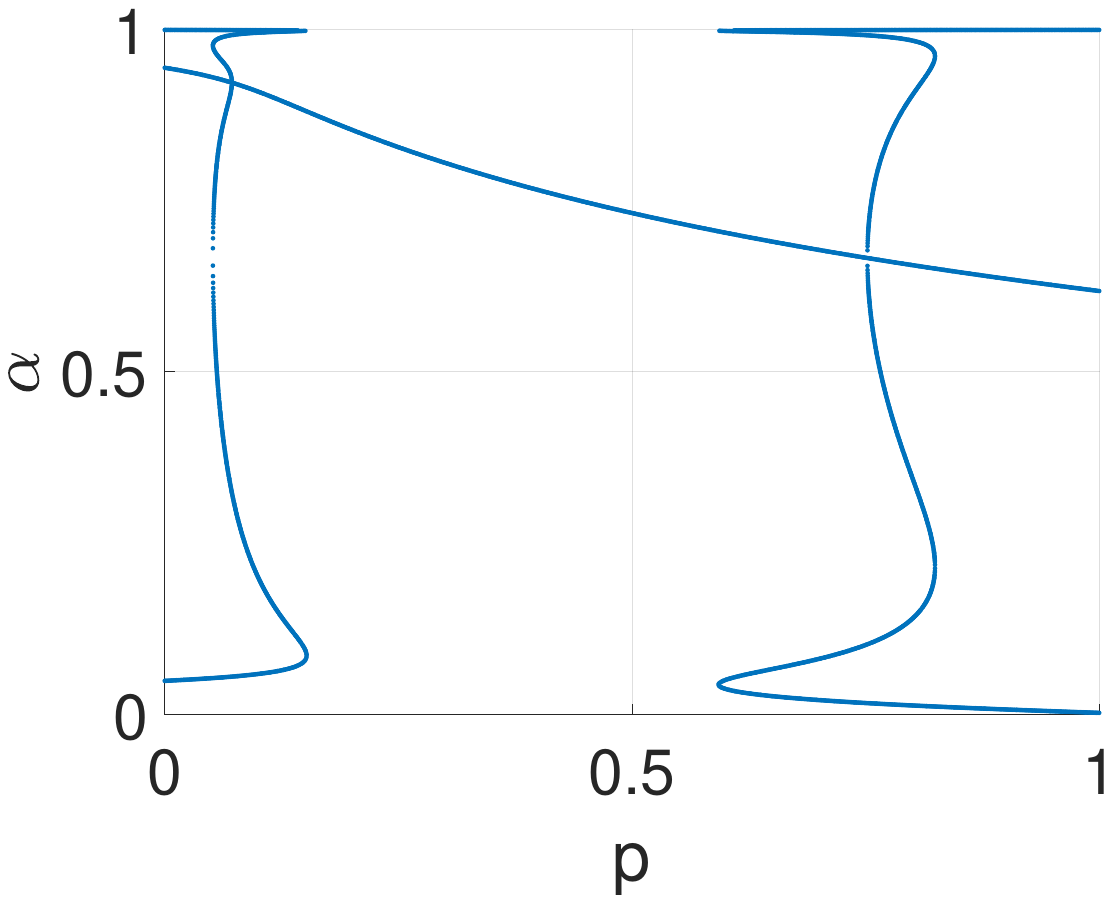}
		\end{subfigure}
		\hfill
		\begin{subfigure}[]{0.33\textwidth}
			\centering
			\includegraphics[width=\textwidth]{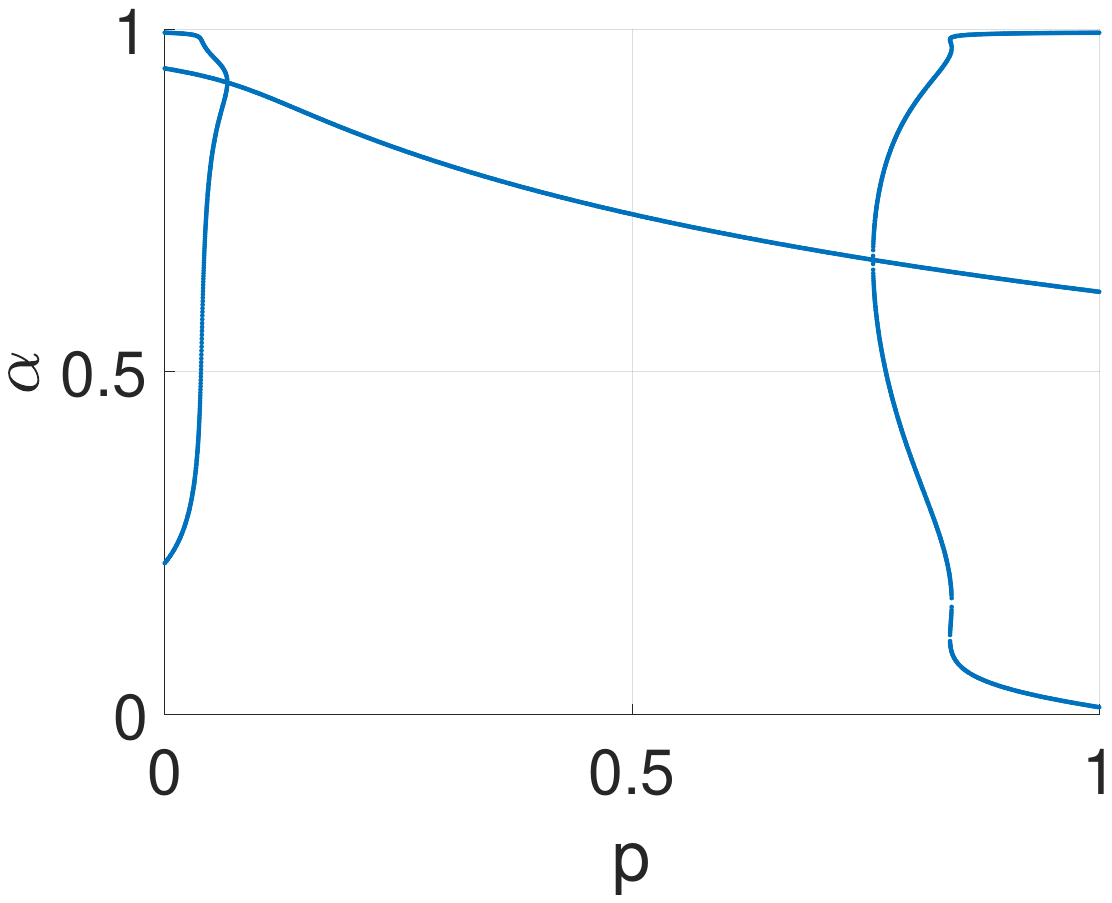}
		\end{subfigure}
		\vspace{-4em}
		
		\caption{Solutions to $\zeta(\alpha)=0$ defined in \eqref{zeta_func} for various $p$ at three values of $q$, with $\hat{\phi}(\alpha)=q+(1-q)(p \alpha^2+(1-p)\alpha^{50})$. \textbf{Left:} $q=0$. \textbf{Middle:} $q=0.001$. \textbf{Right:} $q=0.005$.}
		\label{fig_fixed-p-xi}
	\end{figure}
	\begin{theorem}\label{thm_KSRT}
		Let ~$\CM_{n\geq 0}$ be a sequence of configuration model that satisfies condition \ref{CM_cond} with the limiting degree distribution $(p_k)_{k\geq 1}$ such that $p_1 >0$ and $\sum_{k\geq 0} k(k-2)p_k > 0$. If $\ph'(\alpha_*)\ph'(\alpha^*)<1$, then the size of the Karp-Sipser core of ~$\CM_n$ denoted as $\ks$ satisfies
		\begin{align}\label{thm_main}
			\lim_{n\to\infty} \frac{|\ks|}{n} \xrightarrow{p} \phi(\alpha^*)-\phi(\alpha_*)-(\alpha^*-\alpha_*)\phi'(\alpha_*) 
		\end{align}
	\end{theorem}
	Clearly, the r.h.s. of \eqref{thm_main} is the probability that the root of the random graph $\mathcal{G}\sim \CM_n$ survives in the Karp-Sipser core of $\mathcal{G}$. 
	The proof of the \Thm~\ref{thm_KSRT} relies on the mechanism of an enhanced version of Warning Propagation together with \emph{'Node Labeling Propagation'} and the fixed point analysis of the $\zeta(\alpha)$ function.
	
	\begin{figure}[h!]
		\centering
		\begin{subfigure}[t]{0.33\textwidth}
			\centering
			\includegraphics[width=\textwidth]{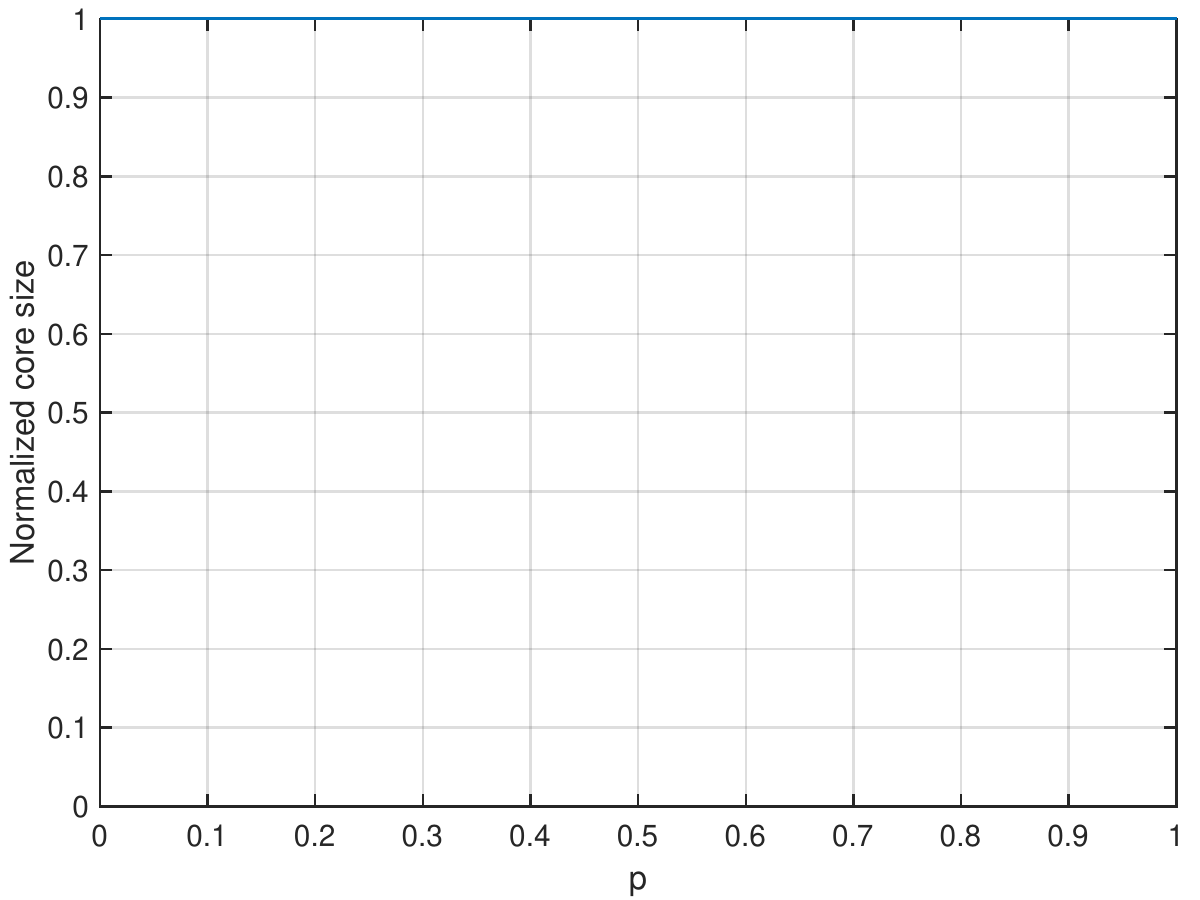}
		\end{subfigure}
		\hfill
		\begin{subfigure}[t]{0.33\textwidth}
			\centering
			\includegraphics[width=\textwidth]{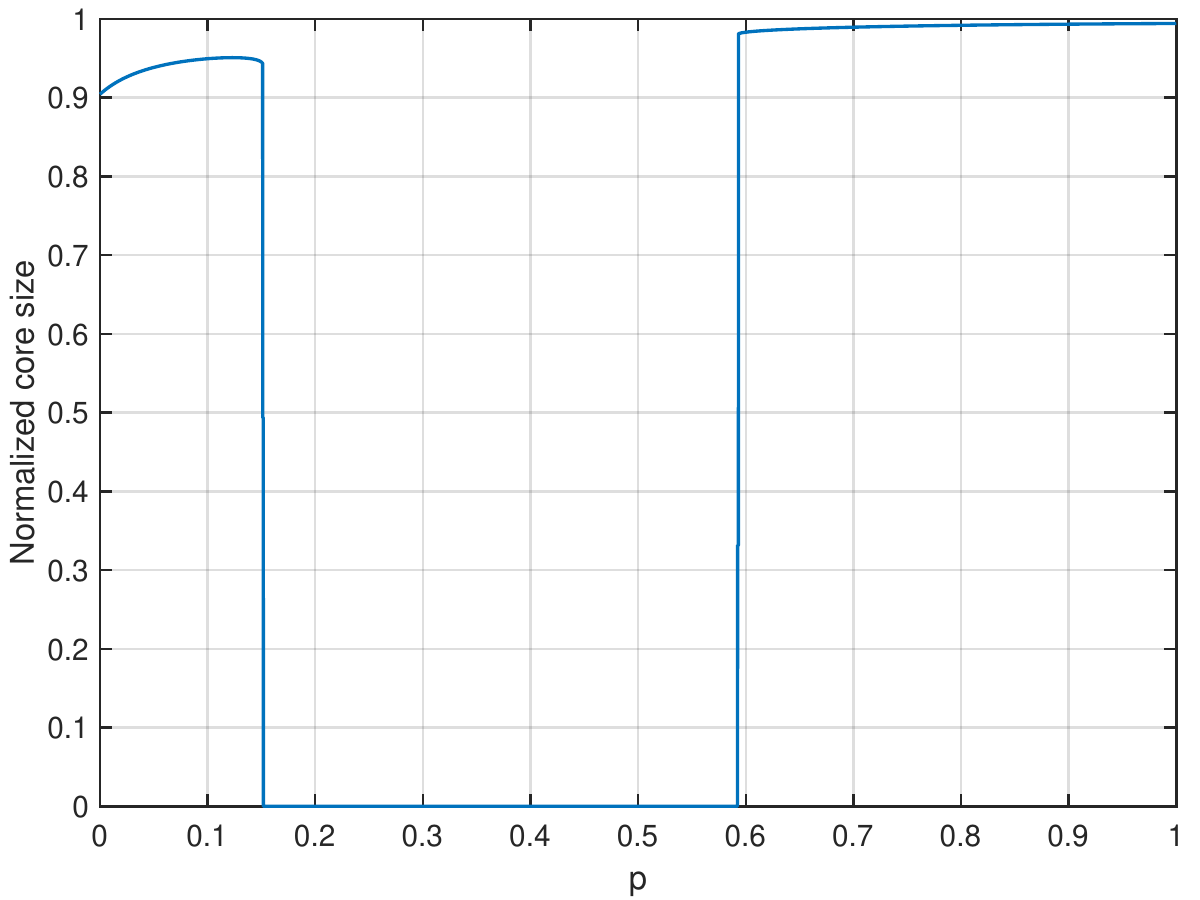}
		\end{subfigure}
		\hfill
		\begin{subfigure}[t]{0.33\textwidth}
			\centering
			\includegraphics[width=\textwidth]{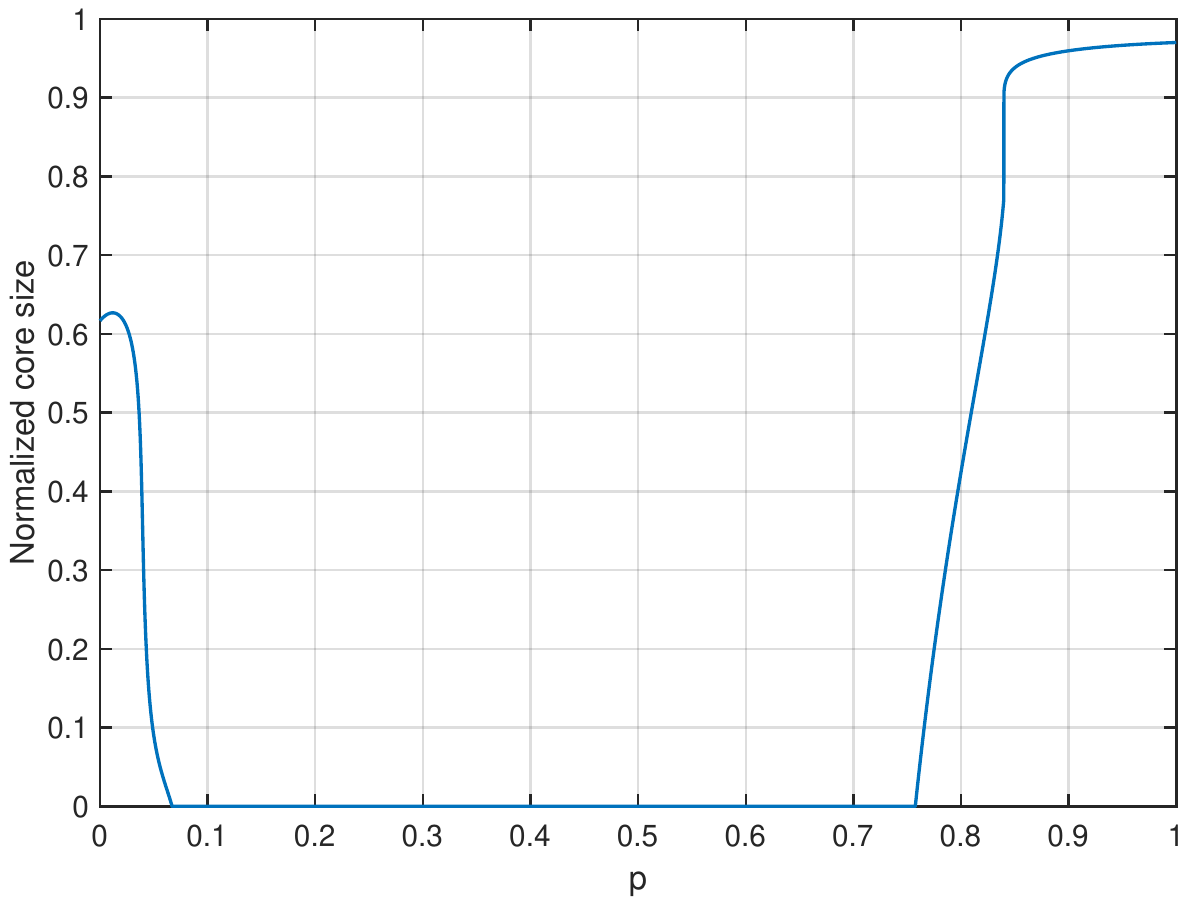}
		\end{subfigure}
		\vspace{-2em}
		
		\caption{The normalized core size for various $p$ at three values of $q$, with
			$\phi(\alpha) =
			\frac{q + (1-q)\left( p\,\alpha^{3}/3 + (1-p)\,\alpha^{51}/51 \right)}
			{q + (1-q)\left( p/3 + (1-p)/51 \right)}$.
			\textbf{Left:} $q=0$. \textbf{Middle:} $q=0.001$. \textbf{Right:} $q=0.005$.
		}
		\label{fig_coresize}
	\end{figure}
	In \Cref{fig_coresize}, the choice of $\phi$ ensures that $\hphi$ in \Cref{fig_fixed-p-xi} satisfies  
	$
	\hphi(\alpha) = \frac{\phi'(\alpha)}{\phi'(1)}$,
	which is consistent with our definition in \eqref{nr_gen_func}. When $q=1$, all vertices are in the Karp-Sipser core by definition. However, when a small proportion of degree-$1$ vertices is introduced into the graph, the normalized size of the Karp-Sipser core may behave quite differently depending on the value of $p$: for some choices of $p$, the core size remains almost unchanged, whereas for others it vanishes abruptly. The latter scenario occurs for intermediate choices of $p$. 
	To reduce the core size, the Karp–Sipser algorithm requires, on the one hand, that small-degree vertices can be easily transformed into leaves, and, on the other hand, that high-degree vertices, once they become neighbors of a leaf, cause the removal of many edges, thereby creating additional leaves.
	\section{Overview}\label{Sec_Overview}
	This section provides an overview of the proofs of \Thm~\ref{thm_KSRT}. 
	We state some intermediate results with the help of fixed point argument and Warning Propagation that lead to the main theorem.
	In the final paragraph, we conclude with a discussion of further related work.
	We also detail where in the following sections the proofs of these intermediate results can be found.
	Throughout we assume $p_1>0$ and $\sum_{k\geq 0} k(k-2)p_k > 0$ and the random graph model $\mathcal{G} (V,E)$ follows our configuration model $\CM_n$. 
	
	\subsection{Fixed Point analysis}
	In this section we will analyze the fixed points of $\zeta(\alpha)$ in $[0,1]$ i.e., the solutions to the equation
	\begin{align}\label{zeta_eq}
		\zeta=\zeta(\alpha)=\alpha+\ph(1-\ph(\alpha))-1=0
	\end{align}
	Equivalently, we set
	\begin{align}\label{xi_func}
		\xi=\xi(\alpha)=1-\hphi(1-\hphi(\alpha))
	\end{align}
	the fixed point equation is given by, $\xi(\alpha)-\alpha=0$, where, $\alpha\in[0,1]$.
	So, the fixed point of $\xi$ coincides with the solutions of $\zeta(\alpha)=0$.
	Below, we will establish the existence, uniqueness or multiplicity of the solutions and their stability by exploiting the convexity and monotonicity properties of $\hphi$ and the routine calculus as well as the standard branching process fixed-point arguments.
	Figure~\ref{fig_fixed-p-xi} shows the solution to equation $\xi(\alpha)-\alpha=0$ or equivalently $\zeta(\alpha)=0$ with different values of $p$-value.  
	
	\begin{lemma}\label{zeta_func_existence}
		Under the assumption $p_1>0$, we have $\zeta(0)<0<\zeta(1)$ and hence by continuity there exists at least one $\alpha_s\in (0,1)$ with $\zeta(\alpha_s)=0$ such that,
		\begin{align*}
			\alpha_s+\hphi(1-\hphi(\alpha_s))=1.
		\end{align*}
	\end{lemma}
	Analogously, the map $\xi(\alpha)=1-\hphi(1-\hphi(\alpha))$ has at least one fixed point $\alpha_s\in (0,1)$ of the equation $\xi(\alpha)-\alpha=0$.

	\begin{proof}
		Let $\lambda=\phi'(1)$. Note that function $\zeta(\alpha)$ is differentiable and continuous on $(0,1)$. Then, 
		\begin{align}\label{zeta_func_zero_one}
			\zeta(0)&=\hphi(1-\hphi(0))-1=\hphi(1-p_1/\lambda)-1,\\
			\zeta(1)&=1+\hphi(1-\hphi(1))-1=p_1/\lambda
		\end{align}
		Since $\hphi(\alpha)$ is strictly increasing, $\hphi(1-p_1/\lambda)<\hphi(1)=1$ and also $p_1/\lambda \ge 0$. 
		Therefore, $\zeta(0)<0$ and $\zeta(1)>0$.
		So, the number of zeros must be odd.
		By the intermediate value theorem, there exists some $\alpha_s\in (0,1)$ such that $\zeta(\alpha_s)=0$.
		Equivalently, the map $\xi(\alpha)=1-\hphi(\alpha)$ also has at least one fixed point $\alpha_s \in (0,1)$.
		Therefore, $\xi(\alpha_s)=\alpha_s$.
	\end{proof}
	So, the above proposition tells us that when the leaves are present in the model, we always have a non-trivial fixed point. 
	Thus, an obvious key question is to find the number of such solutions and identify which one is relevant.
	In the remaining section it is convenient to work with the iteration map $\xi(\alpha)$ defined in \eqref{xi_func}, so that the fixed points of $\xi$ coincides with the stationary points of $\zeta(\alpha)=0$.
	
	The next \Prop~\ref{zeta_func_noffixpts_stab} counts the number of fixed points of \eqref{zeta_eq} and provides the detailed analysis on stability of those fixed points. 
	\begin{proposition}\label{zeta_func_noffixpts_stab}	
		Let \( \aa \) and \( \ab \) denote the smallest and largest solutions of the function \( \zeta(\alpha) \), respectively. 
		If \( \hphi'(\aa) \hphi'(\ab) \ne 1 \), then both \( \aa \) and \( \ab \) are stable fixed points of the function \( \xi \) and $\xi'(\alpha)=\hphi'(\aa) \hphi'(\ab)\in(0,1)$ for $\alpha\in \{\aa,\ab\}$. 
	\end{proposition}
	
	\subsection{Warning Propagation}\label{WP}
	One of the most important tools is an enhanced version of the Warning Propagation (WP) message passing algorithm that helps us to analyze the size of the Karp-Sipser core. 
	In general Warning Propagation applied to a graph $\mathcal{G}(V,E)\sim\CM_n$ consists of two type of messages $\mu_{u\to v}$ and $\mu_{v\to u}$ for each edge $(uv)$ of $\mathcal{G}$.
	These messages can take values in a finite alphabet $\Omega$.
	Moreover let $\mathcal{W}(\mathcal{G})$ be the set of all vectors
	\begin{align}\label{WP_vec}
		\bigl(\mu_{u\to v}\bigr)_{(u,v)\in V(\mathcal{G})^{2}:uv\in E(\mathcal{G})} \in \Omega^{2\abs{E(\mathcal{G})}}
	\end{align}
	Based on some fixed update rules the messages get updated in parallel.
	Let $\left(\binom{\Omega} d\right)$ be the set of all $d$-ary message sets with elements from $\Omega$ for $d\in \mathbb{N}$.
	Then given any input message sets the update rule $\varphi$ computes an output message and is given by,
	\begin{align}\label{WP_operator}
		\varphi:\bigcup_{d\ge 0} \left(\binom{\Omega} d\right) \to \Omega
	\end{align}
	Thus we define the WP operator on $\mathcal{G}(V,E)\sim\CM_n$ by,
	\begin{align}
		&{WP}_{\mathcal{G}}:\mathcal{W}(\mathcal{G})\to \mathcal{W}(\mathcal{G}) \nonumber \\
		&\mu=\left(\mu_{u\to v}\right)_{(uv)}\mapsto 
		\left(\varphi\left(\{\mu_{w\to u}:(u,w)\in E(\mathcal{G}),w\neq v\}\right)\right)_{(uv)} \label{WP_operator1}
	\end{align}
	So, updating the message from $u$ to $v$ we apply the update rule $\varphi$ to the messages $u$ receives from its other neighbors $w\neq v$.
	This approach exploits message-passing stability properties in locally tree-like structures and applies directly to configuration models satisfying Condition~\ref{CM_cond}.
	\subsubsection{\textbf{WP Update rules}}
	Let $\mathcal{G}(V,E) \sim \CM_{n}(d^{(n)})$ be a random graph drawn from the configuration model with degree sequence $d^{(n)}$.
	We denote the message from node $u$ to node $v$ at time $t$ as $\msg{u}{v}{t}$.
	Further a message $\msg{u}{v}{t}=\textcolor{green}{\textbf{U}}$ indicates that the node $u$ is in the Karp-Sipser core due to unmatched of its neighbor and survives after a few iterations of the leaf removal process.
	On the other hand if the message $\msg{u}{v}{t}=\textcolor{blue}{\textbf{M}}$, then it indicates that the node $u$ is not in the core due to matched as unique neighbors of a leaf node.
	
	Our enhanced WP algorithm associates a pair of $\{\textcolor{red}{\textbf{L}},\textcolor{blue}{\textbf{M}},\textcolor{green}{\textbf{U}}\}$-valued messages with every edge of $\mathcal{G}\sim\CM_{n}$.
	Hence, $\mathcal{W(G)}$ be the set of all vectors defined in \eqref{WP_vec} with entries $\mu_{u\to v}=\{\textcolor{red}{\textbf{L}},\textcolor{blue}{\textbf{M}},\textcolor{green}{\textbf{U}}\}$.
	At time $t=0$, we initialize all the messages $\msg{u}{v}{0} = \textcolor{green}{\textbf{U}}$ for all $(u,v)\in E(\mathcal{G})$. 
	\begin{definition}\label{def:core-wp}   
		Throughout the rest of this paper, we define
		\begin{align*}
			\varphi\left(\{\mu_{w\to u}:(u,w)\in E(\mathcal{G}),w\neq v\}\right) =\begin{cases}
				\msgL& \text{if }\ \forall w: \mu_{w\to u} =\msgM\\
				\msgM & \text{if } \abs{\cbc{w: \mu_{w\to u}= \msgL}}\geq 1\\
				\msgU & \text{otherwise.}
			\end{cases}.    
		\end{align*}
		And we define
		\begin{equation}\label{WP_update_rule}
			\mu^{(t+1)}_{u\to v}=\varphi\left(\{\mu^{(t)}_{w\to u}:(u,w)\in E(\mathcal{G}),w\neq v\}\right).
		\end{equation}
	\end{definition}

	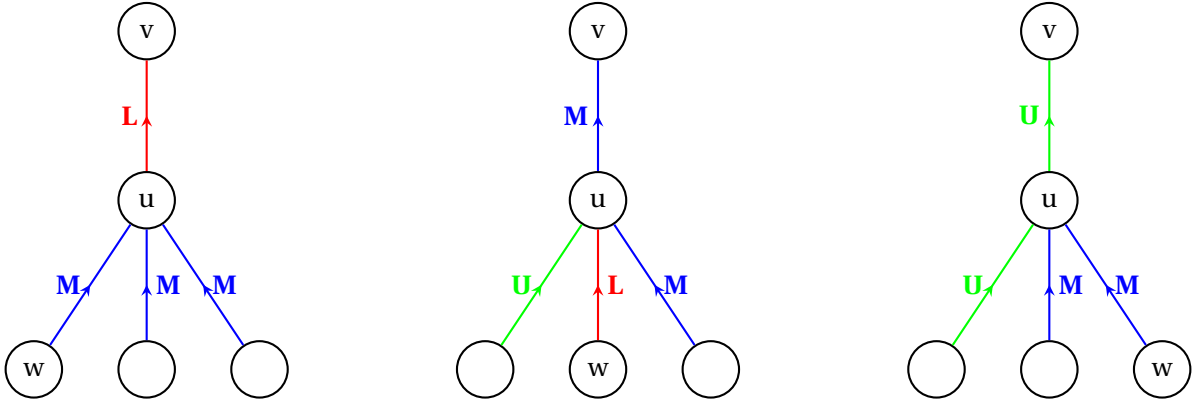
\begin{figure}[t!]
		\centering
		\begin{tikzpicture}[
			L/.style={circle, draw, minimum size=7.5mm, inner sep=0pt, text=black,line width=0.8 pt},
			midarrow/.style={
				postaction={decorate},
				decoration={
					markings,
					mark=at position 0.5 with {\arrow{stealth}}
				}
			}
			]
			
			\node[L] (L1) at (0,4.5) {v};
			\node[L] (L2) at (0, 2.25) {u};
			\node[L] (L3) at (-1.5, 0) {w};
			\node[L] (L4) at (0, 0) {};
			\node[L] (L5) at (1.5, 0) {};
			\draw[red,thick,midarrow] (L2) -- (L1) node[midway,left]
			{\msgL};
			\draw[blue, thick, midarrow] (L3) -- (L2) node[midway, left] {\msgM};
			\draw[blue, thick, midarrow] (L4) -- (L2) node[midway, right] {\msgM};
			\draw[blue, thick, midarrow] (L5) -- (L2) node[midway, right] {\msgM};
			
			\node[L] (L6) at (6,4.5) {v};
			\node[L] (L7) at (6, 2.25) {u};
			\node[L] (L8) at (4.5, 0) {};
			\node[L] (L9) at (6, 0) {w};
			\node[L] (L10) at (7.5, 0) {};
			\draw[blue,thick,midarrow] (L7) -- (L6) node[midway,left]
			{\msgM};
			\draw[green, thick, midarrow] (L8) -- (L7) node[midway, left] {\msgU};
			\draw[red, thick, midarrow] (L9) -- (L7) node[midway, right] {\msgL};
			\draw[blue, thick, midarrow] (L10) -- (L7) node[midway, right] {\msgM};
			
			\node[L] (L11) at (12,4.5) {v};
			\node[L] (L12) at (12, 2.25) {u};
			\node[L] (L13) at (10.5, 0) {};
			\node[L] (L14) at (12, 0) {};
			\node[L] (L15) at (13.5, 0) {w};
			\draw[green,thick,midarrow] (L12) -- (L11) node[midway,left]
			{\msgU};
			\draw[green, thick, midarrow] (L13) -- (L12) node[midway, left] {\msgU};
			\draw[blue, thick, midarrow] (L14) -- (L12) node[midway, right] {\msgM};
			\draw[blue, thick, midarrow] (L15) -- (L12) node[midway, right] {\msgM};

		\end{tikzpicture}
		\caption{A snapshot of Warning Propagation update rules defined in \eqref{WP_update_rule}}
		\label{fig_WP_update}
	\end{figure}
	This rule reflects the core's stability condition: a vertex $u$ send a message $\msgU$ to $v$ only if atleast one of its other neighbors $w \in \partial u \setminus \{v\}$ send message $\msgU$ to $u$ and no $\msgL$ message is received by vertex $v$ from $u$ at time $t$.
	In other words, if node $u$ is a leaf or has any neighbor that is not in the core, it must also leave the core. 
	Hence, 
	\begin{itemize}
		\item A message $\msgL$ from $u$ (to $v$) indicates that $u$ will be removed by Algorithm~\ref{alg_karpsipser} as a leaf and therefore will not belong to the Karp–Sipser core.  
		\item A message $\msgM$ from $u$ indicates that $u$ will be removed by Algorithm~\ref{alg_karpsipser} as a neighbor of a leaf and therefore will not belong to the Karp–Sipser core.  
		\item A message $\msgU$ from $u$ indicates that $u$ remains a possible candidate for the Karp–Sipser core.
	\end{itemize}
	
	\begin{remark}\label{lem_hierarchy}
		To describe the dynamics of the message update, let us put a non-numerical hierarchy to the types of messages as follows:
		\[\msgL<\msgM<\msgU\]
	\end{remark}
	
	In most of the applications of Warning Propagation (WP) the update rule $\varphi$ defined in $\eqref{WP_operator}$ enjoys a monotonicity property which ensures that the point-wise limit ${WP}^{\kll}_{\mathcal{G}}\left(\mu^{(0)}\right):=\lim_{t\to\infty}{WP}^{(t)}_{\mathcal{G}}\left(\mu^{(0)}\right)$
	exists for any graph $\mathcal{G}$ with any initialization $\mu^{(0)} \in \mathcal{W(G)}$. 
	\begin{fact}\label{lem_WP_convergence}
		Let $\mathcal{G}(V,E)\sim\CM_{n}$ be any finite undirected graph, and let $(\mu^{(t)}_{u \to v})_{(u,v) \in E}$ be the sequence of WP messages defined by the update rule in \eqref{WP_update_rule} with
		initialization $\mu^{(0)}_{u\to v}=\msgU$. Then for all $(u,v)\in E$, the message sequence is non-increasing:
		\[
		\mu^{(t+1)}_{u \to v} \leq \mu^{(t)}_{u \to v},
		\]
		and converges in at most $|E|$ steps to a fixed point $\mu^{(\infty)}_{u \to v} \in \cbc{\msgL, \msgM, \msgU}$.  
	\end{fact}
	More details of the monotonicity and convergence of messages $\mu^{(t)}_{u \to v}$ can be found in appendix section.
	
	\subsubsection{\textbf{WP Stability}}
	In addition to the proof of our main theorem, we need to define some more definitions along with the assumptions we require for our configuration model $\mathcal{G}\sim\CM_n$ to satisfy.
	As our goal is to study the fixed points of Warning Propagation on our model $\CM_n$ and particularly, the rate of convergence on the corresponding Galton-Watson limiting tree, we are going to prove that under the assumptions of some stability condition on the update rule the WP fixed point can be characterized in terms of local structure only (for more details refer to \Sec~\ref{Sec_CWP}).
	To this end, we need to define a suitable notion of WP fixed point on a random rooted tree.
	Thanks to the recursive nature of the Galton-Watson tree for helping us to take a convenient shortcut instead of taking actual measure-theoretic gymnastics by defining measurable spaces with infinite points.
	Here, our fixed points are just a probability distribution on $\Omega$ such that, if an offspring of a node $u\in V(\mathcal{G})$ in the Galton-Watson rooted tree sends messages independently according to this distribution, then the messages from $u$ to its parent $v$ will have the same distribution, but this is not generally true in the case of our model $\mathcal{G}\sim\CM_n$.
	The following definition formalizes the idea:
	\begin{definition}\label{def_WP_stability_first}
		Let $\Omega = \{\omega_1, \ldots, \omega_r\}$, and let 
		$\Delta = (y_1, \ldots, y_r)$
		be a probability distribution vector on $\Omega$, where $y_i$ denotes the probability of $\omega_i$ for $i = 1, \ldots, r$.
		, and let 
		$\varphi: \bigcup_{d\geq 0} \binom{\Omega}{d} \to \Omega$
		be a WP update rule. Given a degree distribution $(p_k)_{k\geq 0}$, let $\widehat{\vD}$ be a random variable with the size-biased degree distribution $(\biasp_{k-1})_{k\geq 1}$. $(\vx_i)_{i\geq 1}$ be i.i.d.~random variables follows distribution $\Delta$. We define $\Upsilon_{\varphi}=\bc{\PP\bc{\varphi(\vx_1,\ldots,\vx_{\widehat{\vD}})=\omega_i}}_{1\leq i\leq r}$ and let $f^t$ denote the $t^{\text{th}}$ iterated function of $f$.
		\begin{itemize}
			\item We say that $\Delta$ is a \emph{WP limit} if $\Upsilon_{\varphi}(\Delta) = \Delta$.
			\item We say that $\Delta$ is a \emph{stable WP limit} (with respect to a distance) if it is a WP limit and $\Upsilon_{\varphi}$ is a contraction on a neighborhood of $\Delta$ under this distance.
			\item We say that $\Delta$ is the \emph{WP limit of a probability distribution vector} $\Delta_{0}$ on $\Omega$ if $\Delta$ is a WP limit and, moreover, 
			$\lim_{t\to\infty} \Upsilon^t_{\varphi}(\Delta_0)=\Delta$.
		\end{itemize}
	\end{definition}
	
	The following lemma provides the explicit form of the function $\Upsilon_{\varphi}$ for our choice of $\varphi$ in \Cref{def:core-wp}.
	\begin{lemma}\label{lem_WP_phi}
		Given a probability distribution $(p_k)_{k\geq 1}$ and $\hphi(q)=\sum_{k\ge 1}\widehat{p}_{k}q^{k-1}$ with edge-biased distribution $(\hat{p}_{k-1})_{k\geq 1}$ from \eqref{edge_biased_dist}, the Warning Propagation update function $\varphi$ on three tuples $\{\msgL,\msgM,\msgU\}$ we have 
		\begin{align}\label{WP_update_phi}
			\Upsilon_{\varphi}\bc{q_{\msgL},q_{\msgM},q_{\msgU}}=\left(\hphi(q_{\msgM}),1-\hphi(q_{\msgM}+q_{\msgU}),\hphi(q_{\msgM}+q_{\msgU})-\hphi(q_{\msgM})\right)
		\end{align}    
	\end{lemma}
	The following ~\Lem~\ref{lem_WP_limit} emphasizes the WP limit concerning the initial probability distribution $\Delta_0$.
	
	\begin{lemma}\label{lem_WP_limit}
		For the limiting probability distribution vector $\Delta$ and the initial distribution $\Delta_0$ with $\Delta=\left(1-\alpha^*,\alpha_*,\alpha^*-\alpha_*\right)$ and $\Delta_0=\left(0,0,1\right)$, $\Delta$ is the WP limit of $\Delta_0$ i.e., 
		\begin{align*}
			\lim_{t\to\infty}\Upsilon_{\varphi}^t(\Delta_0)=\Delta.
		\end{align*}
	\end{lemma}
	To obtain the contraction property in the definition of a stable WP limit in \Cref{def_WP_stability_first}, we introduce the following metric on the space of probability distribution matrices, denoted by $\mathscr{D}(\cdot,\cdot)$. This metric is equivalent to the total variation distance $\dtv(\cdot,\cdot)$ on probability distributions.
	\begin{definition}\label{def_new_distance}
		Let the distance of two probability matrices $\Delta_1=(q_{\msgL},q_{\msgM},q_{\msgU})$ and $\Delta_2=(q'_{\msgL},q'_{\msgM},q'_{\msgU})$ with $q_{\msgL}+q_{\msgM}+q_{\msgU}=q'_{\msgL}+q'_{\msgM}+q'_{\msgU}=1$ on the same set as
		\begin{align*}
			\mathscr{D}(\Delta_1,\Delta_2) = \sqrt[4]{\frac{\hphi'(\aa)}{\hphi'(\ab)}} \abs{q_{\msgL} - q_{\msgL}'} + \sqrt[4]{\frac{\hphi'(\ab)}{\hphi'(\aa)}} \abs{q_{\msgM} - q_{\msgM}'}.
		\end{align*}
		Then the distance $\mathscr{D}(\Delta_1,\Delta_2)$ is equivalent to the total variation distance defined as
		\begin{align}\label{eq-eq-d}
			\dtv(\Delta_1, \Delta_2) = \frac{1}{2} \left( \abs{q_{\msgL} - q_{\msgL}'} + \abs{q_{\msgM} - q_{\msgM}'} + \abs{q_{\msgU} - q_{\msgU}'} \right) \leq \abs{q_{\msgL} - q_{\msgL}'} + \abs{q_{\msgM} - q_{\msgM}'}.
		\end{align}
	\end{definition}
	The following lemma illustrates that the WP update function $\Upsilon_{\varphi}$ is a contraction in the neighborhood of the $\Delta_0$. 
	\begin{lemma}\label{lem_WP_stability_limit}
		There exists $\delta>0$ such that for any constants $\varepsilon_1,\varepsilon_2\in \left(-\delta,\delta\right)$ the following satisfies with respect to distance $\mathscr{D}$.
		\begin{align*}
			\mathscr{D}(\Upsilon_{\varphi}(\Delta),\Upsilon_{\varphi}(\Delta+\left(\varepsilon_1,-\varepsilon_1-\varepsilon_2,\varepsilon_2\right)))
			\leq
			(1-\delta)\mathscr{D}(\Delta,\Delta+\left(\varepsilon_1,-\varepsilon_1-\varepsilon_2,\varepsilon_2\right))
		\end{align*}
	\end{lemma}
	
	\begin{theorem}[modified version of \Thm~1.3,\cite{WP_new}]\label{thm_WP_Joon}
		Let $\mathcal{G}\sim\CM_n$ be a sequence of configuration models satisfying Condition \ref{CM_cond} and let $\Delta_0$ be the initial message distribution on $\Omega$.
		Then for any $\varepsilon,\delta>0$, there exists $t_0=t_0(\delta)$ such that for all $t\ge t_0$ we have
		\begin{align}\label{eq-conv-to-t-eps}
			\Pr\left[ \sum_{(u,v) \in E(\mathcal{G})} \mathbf{1}\left\{ \mu_{u \to v}^{(t)} \neq \mu_{u \to v}^{(t_0)} \right\} > \delta n \right] < \varepsilon,
		\end{align}
		where convergence is measured with respect to the distance $ \mathscr{D} $ according to \Def~\ref{def_new_distance}, and $ \mu_{u \to v}^{(t)} $ denotes the WP message after $ t $ rounds.
	\end{theorem}
	In other words, the Warning Propagation messages are identical at any time $t\ge t_0$ to those at time $t_0$ except on a set of at most $\delta n$ directed edges.
	Thus, the Warning Propagation converges rapidly under the mild stability conditions.
	Note that, under the Condition~\ref{CM_cond}, the result of \Thm~\ref{thm_WP_Joon} can be used beyond the sparse random graph model $\mathbb{G}(n,d/n)$ for any general configuration model.
	Technically, the number of steps $t_0$ of Warning Propagation (before which WP stabilizes) does not depend on the parameter $n$ or, even not depends on the exact nature of the configuration model $\CM_n$, it only depends on the desired accuracy $\delta$, the WP update rule $\varphi$ and the initial probability distribution $\Delta_0$. 
	\begin{proof}[Proof of \Cref{thm_WP_Joon}]
		Note that \cite[\Thm~1.3]{WP_new} remains valid when the distance $\dtv$ is replaced by an equivalent distance $\mathscr{D}$. Moreover, \cite[Assumption~2.10]{WP_new} holds for the configuration model with bounded degree (see \cite[Remark 2.15]{WP_new}). Since \Cref{lem_WP_limit,lem_WP_stability_limit} yields that $\Delta=\left(1-\alpha^*,\alpha_*,\alpha^*-\alpha_*\right)$ is a stable WP limit of $\Delta_0=(0,0,1)$ with respect to $\mathscr{D}$, \eqref{eq-conv-to-t-eps} follows from\cite[\Thm~1.3]{WP_new}.
	\end{proof}
	
	\subsection{Karp-Sipser Leaf Removal Process}
	Inspired by the seminal paper \cite{KarpSipser} of Karp and Sipser in 1981, in our first result, we analyze the asymptotic behavior of the size of the Karp-Sipser core in the general configuration model. 
	More specifically, we provide an asymptotic bound on the size of the core in an infinite rooted Galton-Watson tree which serves as a local limit of $\CM_n$.
	It is a greedy method for approximating a maximum matching in sparse random graphs.
	Its first phase, so-called "\textit{leaf-removal phase}" is of particular interest in studying the structure of the sparse random graphs.
	Based on the local structure of the graph, this phase operates via a deterministic labeling procedure.
	Given a finite undirected graph $\mathcal{G}=(V,E)$, the leaf removal process proceeds iteratively below.
	\begin{enumerate}[label=(\roman*).]
		\item \textit{Identify leaves.} Find all vertices $v\in V$ in graph $\mathcal{G}$ whose degree is one. These vertices are labeled as leaves.
		\item \textit{Remove leaves and their neighbors.} For each leaf $v\in V$ in $\mathcal{G}$, remove both $v$ and its unique neighbor $u$.
	\end{enumerate}
	Repeat steps (i)-(ii) iteratively until no leaves remain in the graph.
	The remaining non-isolated vertices in the graph $\mathcal{G}$ that are never removed during this process form the so-called \emph{Karp--Sipser core}.
	This core leads to obstruction to greedy matchings that connect to percolation thresholds in the sparse random graph case.
	In the next subsection, we analyze the labeling propagation of the leaf removal process in the Galton-Watson rooted tree, which serves as the local limit of $\CM_n$.
	\subsubsection{\textbf{Node Labeling Propagation}}
	In \Cref{fig_WP_update}, a vertex $u$ sends a $\msgU$ message to its parent $v$ if it receives at least one $\msgU$ from its children and no $\msgL$, suggesting that $u$ may belong to the Karp–Sipser core. However, once the information from its parent $v$ is also considered, $u$ must receive at least two $\msgU$ messages to remain possible in the Karp–Sipser core. We can model it by defining the labeling propagation process on the node.
	
	We define the labeling propagation process $\mathcal{L}:V(\mathcal{G})\to \{\msgM, \msgU\}$, where
	\begin{itemize}
		\item $\msgM$ is labeled when it receives at most one information of $\msgU$ or an information of $\msgL$, meaning that it will be removed in Algorithm \ref{alg_karpsipser} and not in the Karp–Sipser core.
		\item $\msgU$ is labeled when it receives at least two information from $\msgU$ and no information from $\msgL$, meaning that it remains possible in the Karp–Sipser core.
	\end{itemize}
	As our main agenda is to identify whether a node belongs to the Karp-Sipser core or not after a finite number of iterations, we use the Warning propagation induced \emph{Node Labeling Propagation} for identifying and then further computing the survival probability of the root whether it belongs to the core or not. 
	Even if the vertices are labeled with three labels $\cbc{\msgL,\msgM,\msgU}$, according to the Karp-Sipser leaf removal process in a rooted limiting tree, we are mainly interested in whether a vertex $v\in V$ is labeled as $\msgU$ or not. So, we define the WP-induced Node Labeling Propagation as follows:
	\begin{definition}\label{WP_labelprop}
		Given the messages  $\mu_{u \to v}^{(t)} \in \{\msgL, \msgM, \msgU\}$, we define the label of a node $v\in V$ after $t$ rounds of WP as follows:
		\[
		\mu_v^{(t)} = 
		\begin{cases}
			\msgM & \text{if }~\exists u \in \partial v:\ \mu_{u \to v}^{(t)} = \msgL, \text{or there is at most one } u \in \partial v:\ \mu_{u \to v}^{(t)} = \msgU,\\
			\msgU & \text{otherwise}.
		\end{cases}
		\]
	\end{definition}
	\begin{figure}[]
		\centering
		\begin{tikzpicture}[
			L/.style={circle, draw, minimum size=7.5mm, inner sep=0pt, text=red,line width=0.8 pt},
			M/.style={circle, draw, minimum size=7.5mm, inner sep=0pt,text=blue,line width=0.8 pt},
			U/.style={circle, draw, minimum size=7.5mm, inner sep=0pt,text=green,line width=0.8 pt},
			midarrow/.style={
				postaction={decorate},
				decoration={
					markings,
					mark=at position 0.5 with {\arrow{stealth}}
				}
			}
			]
			\node[M] (L1) at (0, 2.5) {M};
			\node[M] (M1) at (-1.5, 0) {};
			\node[M] (M2) at (0, 0) {};
			\node[M] (M3) at (1.5, 0) {};
			\draw[green, thick, midarrow] (M1) -- (L1) node[midway, left] {\textcolor{green}{U}};
			\draw[blue, thick, midarrow] (M2) -- (L1) node[midway, right] {\textcolor{blue}{M}};
			\draw[blue, thick, midarrow] (M3) -- (L1) node[midway, right] {\textcolor{blue}{M}};
			
			\node[M] (M4) at (6, 2.5) {M};
			\node[L] (L2) at (4.5, 0) {};
			\node[U] (U1) at (6, 0) {} ;
			\node[M] (M5) at (7.5, 0) {};
			\draw[red, thick, midarrow] (L2) -- (M4) node[midway, left] {\textcolor{red}{L}};;
			\draw[green, thick, midarrow] (U1) -- (M4) node[midway, right] {\textcolor{green}{U}} ;
			\draw[green, thick, midarrow] (M5) -- (M4) node[midway, right] {\textcolor{green}{U}};
			
			\node[U] (U2) at (12, 2.5) {U};
			\node[U] (U3) at (10.5, 0) {}; 
			\node[M] (M6) at (12, 0) {};
			\node[U] (U4) at (13.5, 0) {};
			\draw[green, thick, midarrow] (U3) -- (U2) node[midway, left] {\textcolor{green}{U}};
			\draw[blue, thick, midarrow] (M6) -- (U2) node[midway, right] {\textcolor{blue}{M}};
			\draw[green, thick, midarrow] (U4) -- (U2) node[midway, right] {\textcolor{green}{U}};	
		\end{tikzpicture}
		\caption{Node labeling with label $\{\textcolor{blue}{M},\textcolor{green}{U}\}$ as defined in \Def~\ref{WP_labelprop} after running the WP-updation rule \eqref{WP_update_rule}}
		\label{fig:enter-label}
	\end{figure}
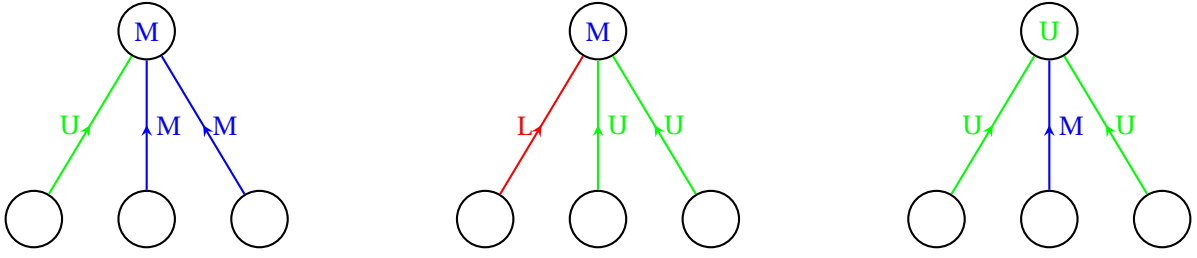
	
	\begin{remark}\label{re-wp-ks}
		We note that $\mu_{v}^{(2t)} = \msgU$ if and only if $v \in V(\mathcal{G})$ belongs to $E$ after the $t^{\text{th}}$ round of step~(\ref{it_2}) in Algorithm~\ref{alg_karpsipser}. If Algorithm~\ref{alg_karpsipser} executes fewer than $t$ rounds of step~(\ref{it_2}), then $E$ is taken to be the final set of vertices at the end, i.e. the vertices in the Karp--Sipser core.
		
	\end{remark}
	
	\subsection{Discussion}
	The main agenda of the present work is to analyze the asymptotic size of the Karp-Sipser core in general configuration model.
	In other words, it calculates the survival probability of the root being in the Karp-Sipser core after $t$ rounds of Warning Propagation message passing process.
	
	Since early 2000s, there has been a lot of literature discussing the importance of Karp-Sipser core in random graphs.
	Bauer and Golinelli in \cite{BGC} provides a detailed analysis of Karp-Sipser core by stating that the choice of choosing vertices to remove in first step is independent of the core which gives the intrinsic property of the graph.
	Coming to the applications of Karp-Sipser core in understanding large scale networks, not only limited to find near maximum matching but also use in finding independent set \cite{GNS} and the rank of adjacency matrices \cite{BLS,HMZ,HMZ2}.
	Nevertheless, most of the previous papers rely only on the sparse \Erdos-\Renyi~random graphs or some log-concave property because of its nice phenomena that the Karp-Sipser core has almost perfect matching in this model \cite{BLS,BLS2}.
	Although some of the papers address the detailed structure of the Karp–Sipser core in configuration models, their model only allows vertex degree no more than $3$  \cite{BCC}.
	Beside, Budzinski and Alice \cite{BCC} paper depends on the differential equation method.
	Instead here we analyze a limiting distribution of the Karp-Sipser core in configuration model assuming the regularity conditions on the degree distribution and the presence of leaves.
	More precisely, we investigate the corresponding Galton-Watson tree which is the local limit of configuration model and can be formalized as 'local weak convergence' \cite{AS,BordenaveCaputo,BS}.
	
	The most important tool we employ in this paper to analyze the asymptotic size 
	of the Karp--Sipser core is an amplified version of Warning Propagation. 
	Earlier, in 2017, Coja-Oghlan et al. \cite{CCKS} applied Warning Propagation 
	to study the local structure of the $k$-core of a random graph, where they observed that ``a bounded number of iterations suffice to obtain an accurate approximation of the $k$-core'', and moreover that such bounded iterations in a random graph can be characterized by the same process in its local limit. 
	Subsequently, in \cite{CCKS2}, they established a multivariate local limit theorem for both the order and the size of the $k$-core of a random graph, again via Warning Propagation. 
	
	In \cite{WP_new}, Cooley et al. further extended this line of work by providing a detailed analysis of the Warning Propagation message-passing algorithm, showing that it applies not only to the $k$-core of a random graph, 
	but also to a wide class of combinatorial recursive procedures, 
	including Unit Clause Propagation.  
	Furthermore, based on some stability conditions they proposed a generalization of Warning Propagation on multi-type random graph models.
	They showed that under mild assumptions on the random graph model and the stability of the message limit, the Warning Propagation converges rapidly.
	In effect, their fixed point analysis of the message passing process on random graph boils down to the analysis of the procedure on a multi-type Galton Watson tree.
	This multi-type Galton Watson branching process builds a connection between the local structure of the graph and the Warning Propagation messages.
	In random combinatorial structures, one of the most common technique is to analyze the local structure of the underlying graphical model given the message distributions.
	In \cite{Chatterjee}, they also provide a detailed analysis of local weak convergence and its connection to Warning Propagation on the random $k$-XORSAT formula which can be translated into a linear system of equations over $\mathbb{F}_2$.
	A similar kind of analysis has been found in \cite{CCKLR2} where they analyzed the local structure of the graph that represents the matrix over $\mathbb{F}_2$.
	Their corresponding branching process consists of the $\Po(d)$ number of offspring as they analyzed the process on sparse \Erdos-\Renyi~ random graph.
	We here instead make an analysis on more generalized random graph model with multinomial branching process argument.
	
	In addition to this, in \cite{CCKLR2} their analysis on the stability of Warning Propagation limit acts on a bipartite graph consisting of variable and check nodes and for proving the stable limit of the initial message distribution they use the total variation distance.
	Instead in this paper we introduce a new distance metric defined in \eqref{def_new_distance} and show that under the new distance rather than the total variance distance used in \cite{WP_new} and the stability assumptions, their result remains valid.
	Furthermore, we proved that under the new distance metric the limiting final message distribution at the end of Warning Propagation process is the stable WP limit of the initial message distribution which ensures the rapid convergence guarantee of Warning Propagation.
	
	\subsection{Preliminaries and Notation}
	Throughout we assume $p_1>0$ and $\sum_{k\geq 0} k(k-2)p_k > 0$ and also the graph $\mathcal{G}(V,E)$ follows a sequence of configuration model $\CM_{n\ge 0}$.
	Unless specified otherwise we tactically assume that the number of nodes $n$ in the graph $\mathcal{G}$ is sufficiently large for the analysis of the asymptotic size of the Karp-Sipser core.
	Asymptotic notation such as $O(\cdot)$ refers to the limit of large $n$ by default.
	We continue to denote $\aa$ and $\ab$ are the smallest / largest fixed points of $\zeta(\alpha)$ in $[0,1]$.
	The labels $\msgL,\msgM,\msgU$ denote the messages passing from one node to other according to the Warning Propagation rules defined in \ref{WP_update_rule} where $\msgL$ stands for leaf nodes in the graph, $\msgM$ stands for the unique neighbor of the leaves due to matched and $\msgU$ denotes the nodes which are unmatched and survive till the end and belongs to the Karp-Sipser core.
	\subsection{Organization}
	In the remaining sections we work our way to the proof of our main result \Thm~\ref{thm_KSRT}.
	Specifically, in \Sec~\ref{sec_prop2.3} we prove \Prop~\ref{zeta_func_noffixpts_stab} which confirms the stability of the fixed points of the function $\zeta$ defined in \eqref{zeta_eq}.
	In \Sec~\ref{Sec_CWP} we provide a detailed analysis of the enhanced version of Warning Propagation process which contains the proof of \Lem~\ref{lem_WP_limit} and \Lem~\ref{lem_WP_stability_limit} which are the most supportive lemmas to prove our main result.
	Finally, \Sec~\ref{sec_final} gives us the proof of our main result \Thm~\ref{thm_KSRT} by providing a detailed analysis of Warning Propagation along with Node Labeling mechanism first in Galton-Watson tree followed by configuration model.
	\section{Proof of \Prop~\ref{zeta_func_noffixpts_stab}}\label{sec_prop2.3}
	This section consists mainly of the proof of \Prop~\ref{zeta_func_noffixpts_stab}, i.e. the stability of the fixed points of function $\zeta(\alpha)$ with some intermediate results which will lead to proof of our main \Thm~\ref{thm_KSRT}.
	
	Although some steps are mildly intricate, the proof of \Prop~\ref{zeta_func_noffixpts_stab} mostly consists of 'routine calculus'.
	As a convenient shorthand we compute the first and second order derivatives of the function \eqref{xi_func}.
	Its derivatives read
	\begin{align}\label{xi_first_der}
		\xi'=\xi'(\alpha)=\hphi'(\alpha)\hphi'(1-\hphi(\alpha)) \quad \mbox{and}
	\end{align}
	\begin{align}\label{xi_sec_der}
		\xi''=\xi''(\alpha)=\hphi''(\alpha)\hphi'(1-\hphi(\alpha))-
		(\hphi'(\alpha))^2\hphi''(1-\hphi(\alpha))
	\end{align}
	To prove \Cref{zeta_func_noffixpts_stab}, we first state a lemma showing the duality of the solutions to $\zeta(\alpha) = 0$.
	\begin{lemma}\label{lem:solutions-relation-zeta}
		If $\beta$ is a solution to $\zeta(\alpha)=0$ defined in \eqref{zeta_func} in $[0,1]$, then $1 - \hat{\phi}(\beta)$ is also a solution in $[0,1]$. Specifically, $\aa=1 - \hat{\phi}(\ab)$ and $\ab=1 - \hat{\phi}(\aa)$.
	\end{lemma}
	\begin{proof}
		Let \( \beta \) be a solution to \( \zeta(\alpha) = 0 \). Then,
		\[
		\hphi(1 - \hphi(\beta)) = 1 - \beta.
		\]
		This implies that
		\[
		\zeta(1 - \hphi(\beta)) = 1 - \hphi(\beta) + \hphi(1 - \hphi(1 - \hphi(\beta))) - 1 = -\hphi(\beta) + \hphi(\beta) = 0.
		\]
		Thus, \( 1 - \hphi(\beta) \) is also a solution of \( \zeta \). In particular, since \( \alpha_* \) and \( \alpha^* \) denote the smallest and largest solutions of \( \zeta \), if \( \alpha_* \ne 1 - \hphi(\alpha^*) \), then \( 1 - \hphi(\alpha^*) \) must be a distinct solution with \( 1 - \hphi(\alpha^*) > \alpha_* \). Moreover, applying the same argument again, the value
		\[
		1 - \hphi(1 - \hphi(\alpha^*))
		\]
		is also a solution to \( \zeta \). Since \( 1 - \hphi(\cdot) \) is strictly decreasing, we have
		\[
		\alpha^* = 1 - \hphi(1 - \hphi(\alpha^*)) < 1 - \hphi(\alpha_*).
		\]
		But \( 1 - \hphi(\alpha_*) \) is again a solution to \( \zeta \), contradicting the assumption that \( \alpha^* \) is the largest solution. Therefore, we must have \( \alpha_* = 1 - \hphi(\alpha^*) \). By symmetry, it also follows that \( \alpha^* = 1 - \hphi(\alpha_*) \).	
	\end{proof}
	
	\begin{proof}[Proof of \Cref{zeta_func_noffixpts_stab}]
		Since \( \alpha_* \) and \( \alpha^* \) are solutions to the equation \( \zeta(\alpha) = 0 \), they are also fixed points of the function \( \xi \). Note that,  
		\[
		\zeta(0) = \hphi(1 - \hphi(0)) - 1 < 0,
		\]
		and since \( \alpha_* \) is the smallest solution to \( \zeta(\alpha) = 0 \) in the interval \([0, 1]\), it follows that \( \zeta'(\alpha_*) \ge 0 \).
		
		Observe that the derivative of \( \zeta \) satisfies:
		\[
		\zeta'(\alpha) = -\hphi'(\alpha)\hphi'(1 - \hphi(\alpha)) + 1.
		\]
		Combining the above equation with  Lemma~\ref{lem:solutions-relation-zeta}, we obtain:
		\[
		\zeta'(\alpha_*)=-\hphi'(\alpha_*)\hphi'(\alpha^*) + 1 \geq 0.
		\]
		Since $\hphi'(\alpha_*)\hphi'(\alpha^*)>0$, unless \( \hphi'(\alpha_*) \hphi'(\alpha^*) = 1 \), we have $\hphi'(\alpha_*)\hphi'(\alpha^*)\in(0,1)$, as desired.
	\end{proof}
	
	\section{Chasing Warning Propagation Limit}\label{Sec_CWP}
	
	The proof of \Cref{thm_WP_Joon} requires \Cref{lem_WP_phi,lem_WP_limit,lem_WP_stability_limit}, and in this section we prove them in the order of their appearance.
	
	\begin{proof}[Proof of \Cref{lem_WP_phi}]
		Let $\bc{q_{\msgL}',q_{\msgM}',q_{\msgU}'}=\Upsilon_\varphi\bc{q_{\msgL},q_{\msgM},q_{\msgU}}$. According to the WP updating rule in \eqref{def:core-wp}, we have
		\begin{align}
			q'_{\msgL}&=\sum_{k\ge 1}\hat{p}_k \left(q_{\msgM}\right)^{k-1}
			=\hphi\left(q_{\msgM}\right) \label{eq_recurrence_qL} \\
			q'_{\msgM} &= \sum_{k \ge 1} \biasp_k \sum_{\substack{a + b + c = k - 1 \\ a \ge 1}} \binom{k - 1}{a, b, c} \left(q_{\msgL}^{(t)}\right)^a \left(q_{\msgM}^{(t)}\right)^b \left(q_{\msgU}^{(t)}\right)^c
			=1-\hphi\left(q_{\msgM}+q^{(t)}_{\msgU}\right) \label{eq_recurrence_qM}\\
			q_{\msgU}' &= \sum_{k \ge 1} \biasp_k \sum_{\substack{a + b = k - 1 \\ b \ge 1}} \binom{k - 1}{a, b} \left(q_{\msgM}^{(t)}\right)^a \left(q_{\msgU}^{(t)}\right)^b
			=\hphi\left(q_{\msgM}+q_{\msgU}\right)-\hphi\left(q_{\msgM}\right)\label{eq_recurrence_qU},
		\end{align}
		as desired. 
	\end{proof}
	Given \Cref{lem_WP_phi}, we next investigate the WP limit of $\Delta_0=(0,0,1)$.
	
	\begin{proof}[Proof of \Lem~ \ref{lem_WP_limit}]
		For $\Delta=\left(1-\alpha^*,\alpha_*,\alpha^*-\alpha_*\right)$, it is straight forward to check that $\Upsilon_{\varphi}(\Delta)=\Delta$. We next show that $\Delta$ is the WP limit of the initial message distribution $\Delta_0=(0,0,1)$.
		Recall and the edge-biased distribution $\biasp_k$ from \eqref{edge_biased_dist}.
		Also, let the message distribution after $t+1$ step of WP process $\Delta_{t+1}=\Upsilon_{\varphi}(\Delta_t)$ i.e. applying the WP update rule to the message distribution at time $t$. 
		Then \Cref{lem_WP_phi} give that $\Delta_1=(\biasp_1,0,1-\biasp_1)$.
		Further, by \eqref{eq_recurrence_qM} and \eqref{eq_recurrence_qU} in \Cref{lem_WP_phi},
		\begin{align*}
			q_{\msgM}^{(t+1)}+q_{\msgU}^{(t+1)} = 1-\hphi\bc{q_{\msgM}^{(t)}}.
		\end{align*}
		Hence,
		\begin{align}\label{eq:update-0-2}
			q_{\msgM}^{(t+1)} = 1-\hphi\bc{q_{\msgM}^{(t)}+q_{\msgU}^{(t)}}=1-\hphi\bc{1-\hphi\bc{q_{\msgM}^{(t-1)}}}.
		\end{align}
		Since $q_{\msgM}^{(0)}=q_{\msgM}^{(1)}=0$, we conclude from \eqref{eq:update-0-2} that 
		\begin{align*}
			q_{\msgM}^{(2t+1)}=q_{\msgM}^{(2t)}=1-\hphi\bc{1-\hphi\bc{q_{\msgM}^{(2t-2)}}}.
		\end{align*}
		Note that $\aa>0$ is the smallest solution of $\zeta$ in $[0,1]$. 
		Then $\xi(\alpha)=1-\ph(1-\ph(\alpha))>\alpha$ for $\alpha\in [0,\aa)$, as \Cref{zeta_func_noffixpts_stab} yields that
		\begin{align*}
			\zeta'(\aa)=1-\xi'(\aa)>0.
		\end{align*}    
		Since $1-\ph(1-\ph(\alpha))$ is a strictly increasing function,  $1-\ph(1-\ph(\alpha))<1-\ph(1-\ph(\ab))=\aa$ for $\alpha\in[0,\aa)$. 
		Note $q_{\msgM}^{(0)}=0<\aa$. 
		It is then easy to deduce by induction that 
		\begin{align*}
			q_{\msgM}^{(2t)}\in (q_{\msgM}^{(2t-2)},\aa).
		\end{align*}
		Then, the monotone fixed point theorem yields $q_{\msgM}^{(2t+1)}=q_{\msgM}^{(2t)}\uparrow \aa$ as $t\to\infty$, i.e. $q_{\msgM}^{(t)}\uparrow \aa$.
		Further, \eqref{eq_recurrence_qL} yields $q_{\msgL}^{(t)}=\hphi\bc{q_{\msgM}^{(t-1)}}\uparrow \hphi\bc{\aa}$.
		And thus, $q_{\msgU}^{(t)}=\bc{1-q_{\msgM}^{(t)}-q_{\msgL}^{(t)}}\downarrow \bc{1-\hphi\bc{\aa}-\aa}=\ab-\aa$, as desired.
	\end{proof}
	
	Finally, we prove that the distance $\mathscr{D}$ is a contraction near this WP limit.
	\begin{proof}[Proof of \Lem~\ref{lem_WP_stability_limit}]
		By \Cref{zeta_func_noffixpts_stab}, we have $\hphi'(\aa)\hphi'(\ab)<1$. Hence, there exists a constant $\delta>0$ such that
		\begin{align*}
			\hphi'(\aa)\hphi'(\ab)\max\cbc{\bc{\frac{\hphi'(\aa+\delta)}{\hphi'(\aa)}}^2,\bc{\frac{\hphi'(\ab+\delta)}{\hphi'(\ab)}}^2} \leq 1 - \delta,
		\end{align*}
		which implies that
		\begin{align}\label{eq:upbound-derivative}
			\hphi'(\aa+\delta) \leq (1-\delta)\sqrt{\frac{\hphi'(\aa)}{\hphi'(\ab)}} \quad \text{and} \quad \hphi'(\ab+\delta) \leq (1-\delta)\sqrt{\frac{\hphi'(\ab)}{\hphi'(\aa)}}.
		\end{align}
		To verify that under $\mathscr{D}$, the point $\Delta$ is the stable WP limit of $\Delta_0$, it suffices to show that for any $\eps_1, \eps_2 \in (-\delta, \delta)$,
		\begin{align*}
			\mathscr{D}\bc{\Upsilon_{\varphi}(\Delta), \Upsilon_{\varphi}\bc{\Delta + (\eps_1, -\eps_1 - \eps_2, \eps_2)}} \leq (1 - \delta)\mathscr{D}\bc{\Delta, \Delta + (\eps_1, -\eps_1 - \eps_2, \eps_2)}.
		\end{align*}
		Indeed, since $    \Upsilon_{\varphi}(x,y,z)=\left(\hphi(y),1-\hphi(y+z),\hphi(y+z)-\hphi(y)\right)$ and $\Upsilon_{\varphi}(\Delta) = \Delta$ by \Cref{lem_WP_limit}, it remains to prove that
		\begin{align}\label{eq:expand-dif-dne}
			\sqrt[4]{\frac{\hphi'(\aa)}{\hphi'(\ab)}} \abs{\hphi(\ab) - \hphi(\ab - \eps_2)} + \sqrt[4]{\frac{\hphi'(\ab)}{\hphi'(\aa)}} \abs{\hphi(\aa) - \hphi(\aa + \eps_1)} \leq (1 - \delta)\left[\sqrt[4]{\frac{\hphi'(\aa)}{\hphi'(\ab)}} \abs{\eps_1} + \sqrt[4]{\frac{\hphi'(\ab)}{\hphi'(\aa)}} \abs{\eps_2}\right].
		\end{align}
		By the mean value theorem, the monotonicity of $\hphi'$, and  \eqref{eq:upbound-derivative}, we obtain
		\begin{align}\label{eq-dif-eps1}
			\abs{\hphi(\aa) - \hphi(\aa + \eps_1)} \leq \hphi'(\aa + \delta)\abs{\eps_1} \leq (1 - \delta)\sqrt{\frac{\hphi'(\aa)}{\hphi'(\ab)}}\abs{\eps_1},
		\end{align}
		and
		\begin{align}\label{eq-dif-eps2}
			\abs{\hphi(\ab) - \hphi(\ab - \eps_2)} \leq \hphi'(\ab + \delta)\abs{\eps_2} \leq (1 - \delta)\sqrt{\frac{\hphi'(\ab)}{\hphi'(\aa)}}\abs{\eps_2}.
		\end{align}
		
		Combining \eqref{eq-dif-eps1} and \eqref{eq-dif-eps2} yields \eqref{eq:expand-dif-dne}, as desired.
	\end{proof}
	\section{Proof of \Thm~\ref{thm_KSRT}}\label{sec_final}
	Analyzing the asymptotic behavior of the core size in $\CM_n$
	in the sparse limit, we take advantage of the fact that the local structure of the configuration model $\CM_n$ converges to a Galton-Watson tree.
	Since \ref{thm_WP_Joon} ensures that a finite number of steps of the Warning Propagation process already provide a good estimate of the final outcome, and a finite number of steps of the process on the graph is close to that on its local limit, it is sufficient to study Warning Propagation on the local limit.
	
	Following this idea, in this section we will prove our main theorem by using the concept of Warning Propagation process first applying to the Galton-Watson limiting tree followed by the random graph model $\mathcal{G}\sim\CM_n$.
	\subsection{Warning Propagation in Galton-Watson rooted tree}
	Consider the WP-messages on their local limit, a rooted unimodular Galton–Watson tree $(\mathcal{T}, o)$ with root degree distribution $(p_k)_{k \geq 0}$. 
	The degree distribution of all non-root vertices then follows the size-biased distribution $(\biasp_{k+1})_{k \geq 0}$ defined in \eqref{edge_biased_dist}. 
	On this Galton–Watson tree, we only consider the messages $\tilde\mu_{u \to v}^{(t)}$ where $v$ is the parent of $u$. 
	We initialize $\tilde\mu_{u \to v}^{(0)} = \msgU$ for all such $(u, v)\in E(\mathcal{T})$ pairs, and update the messages according to \Cref{def:core-wp}. 
	Let $x_1, \ldots, x_d$ be the children of the root $o$. 
	Recall that the initial message distribution $\Delta_0 = (0, 0, 1)$ and $\Delta_t=(q_{\msgL}^{(t)},q_{\msgM}^{(t)},q_{\msgU}^{(t)})= \Upsilon_{\varphi}(\Delta_{t-1})$. 
	\begin{claim}\label{Claim_GW}
		The message distribution after $t$ rounds of Warning Propagation process is given by,
		$$\Delta_{t}=\bc{\Pr\bc{\tilde\mu_{x_i \to o}^{(t)}=\msgL},\Pr\bc{\tilde\mu_{x_i \to o}^{(t)}=\msgM},\Pr\bc{\tilde\mu_{x_i \to o}^{(t)}=\msgU}}$$.
	\end{claim}
	\begin{proof}
		We prove this claim by induction. 
		For the base case, by our initialization $\tilde\mu_{x_i \to o}^{(0)} = \msgU~(\forall i\in[d])$,
		$$\Delta_{0}=\bc{\Pr\bc{\tilde\mu_{x_i \to o}^{(0)}=\msgL},\Pr\bc{\tilde\mu_{x_i \to o}^{(0)}=\msgM},\Pr\bc{\tilde\mu_{x_i \to o}^{(0)}=\msgU}},$$
		that is, the claim holds for $t=0$.
		To advance induction, assume that the claim holds for some fixed \(t\), and let \(y_1, \ldots, y_{\partial_i}\) be the children of \(x_i\).
		Also, let \((\mathcal{T}_x, x)\) denote the subtree of \((\mathcal{T}, o)\) rooted at \(x\), i.e. the subtree consisting of \(x\) and all its descendants.
		Then \((\mathcal{T}_{x_i}, x_i)\) and \((\mathcal{T}_{y_j}, y_j)\) are identically distributed. 
		As a consequence, the induction hypothesis yields
		$$\Delta_{t}=\bc{\Pr\bc{\tilde\mu_{y_j \to x_i}^{(t)}=\msgL},\Pr\bc{\tilde\mu_{y_j \to x_i}^{(t)}=\msgM},\Pr\bc{\tilde\mu_{y_j \to x_i}^{(t)}=\msgU}}.$$
		Furthermore, we note that the subtrees \(\{(\mathcal{T}_{y_j}, y_j)\}_{j \in [\partial_i]}\), and hence the messages \(\{\tilde\mu_{y_j \to x_i}^{(t)}\}_{j \in [\partial_i]}\), are mutually independent. 
		Therefore, applying the update rule from \Cref{def:core-wp}, and using the fact that \(\widehat{D}\) follows the size-biased distribution \((\biasp_{k+1})_{k \geq 0}\), we obtain
		\begin{align*}
			\Pr\bc{\tilde\mu_{x_i \to o}^{(t+1)}=\msgL} &= \sum_{k \ge 0} \Pr\bc{\widehat{D}=k} \cdot \left(q_{\msgM}^{(t)}\right)^{k-1}\\
			&=\hphi\bc{q_{\msgM}^{(t)}}=q_{\msgL}^{(t+1)}.\\
			\Pr\bc{\tilde\mu_{x_i \to o}^{(t+1)}=\msgM} &= \sum_{k \ge 0} \Pr\bc{\widehat{D}=k} \sum_{\substack{a + b + c = k \\ a \ge 1}} \binom{k }{a, b, c} \left(q_{\msgL}^{(t)}\right)^a \left(q_{\msgM}^{(t)}\right)^b \left(q_{\msgU}^{(t)}\right)^c\\
			&=1-\hphi\bc{q_{\msgM}^{(t)}+q_{\msgU}^{(t)}}=q_{\msgM}^{(t+1)};\\
			\Pr\bc{\tilde\mu_{x_i \to o}^{(t+1)}=\msgU} &= \sum_{k \ge 0} \Pr\bc{\widehat{D}=k} \sum_{\substack{a + b = k  \\ b \ge 1}} \binom{k }{a, b} \left(q_{\msgM}^{(t)}\right)^a \left(q_{\msgU}^{(t)}\right)^b \\
			&=\hphi\bc{q_{\msgM}^{(t)}+q_{\msgU}^{(t)}}-\hphi\bc{q_{\msgM}^{(t)}}=q_{\msgU}^{(t+1)};
		\end{align*}
		Hence, we conclude that the claim also holds for \(t+1\).
		
		Since both the base case and the induction step have been proved to be true, by mathematical induction the claim holds true for any integer $t\geq 0$.
	\end{proof}
	Now using the concept of Node Labeling Propagation defined in \Cref{WP_labelprop} we are going to prove the probability that the root will survive i.e. belongs to the Karp-Sipser core after $t$ rounds of Warning Propagation process.
	Let $\tilde\mu_o^{(t)}$ be the WP induced Node Labeling Propagation with the rule defined in \Cref{WP_labelprop}. 
	Then,
	\begin{align*}
		\Pr\bc{\tilde\mu_o^{(t)}=\msgU}=&1-\sum_{k \ge 0} \Pr\bc{D=k} \sum_{\substack{a + b + c = k \\ a \ge 1}} \binom{k }{a, b, c} \left(q_{\msgL}^{(t)}\right)^a \left(q_{\msgM}^{(t)}\right)^b \left(q_{\msgU}^{(t)}\right)^c\\
		&-\sum_{k \ge 0} \Pr\bc{D=k} \sum_{\substack{a + b= k \\ b \le 1}} \binom{k }{a, b} \left(q_{\msgM}^{(t)}\right)^a \left(q_{\msgU}^{(t)}\right)^b.
	\end{align*}
	Since $D$ follows the distribution $(p_k)_{k\geq 0}$,
	\begin{align*}
		&1-\sum_{k \ge 0} \Pr\bc{D=k} \sum_{\substack{a + b + c = k \\ a \ge 1}} \binom{k }{a, b, c} \left(q_{\msgL}^{(t)}\right)^a \left(q_{\msgM}^{(t)}\right)^b \left(q_{\msgU}^{(t)}\right)^c=\phi\bc{q_{\msgM}^{(t)}+q_{\msgU}^{(t)}};\\
		&\sum_{k \ge 0} \Pr\bc{D=k}   \left(q_{\msgM}^{(t)}\right)^k =\phi\bc{q_{\msgM}^{(t)}};\\
		&\sum_{k \ge 0} \Pr\bc{D=k}  k q_{\msgU}^{(t)} \left(q_{\msgM}^{(t)}\right)^{k-1}=q_{\msgU}^{(t)}\phi'\bc{q_{\msgM}^{(t)}}.
	\end{align*}
	Consequently,
	\begin{align*}
		\Pr\bc{\tilde\mu_o^{(t)}=\msgU}=\phi\bc{q_{\msgM}^{(t)}+q_{\msgU}^{(t)}}-\phi\bc{q_{\msgM}^{(t)}}-q_{\msgU}^{(t)}\phi'\bc{q_{\msgM}^{(t)}}.
	\end{align*}
	In the next subsection, we will prove that the above phenomenon remains true for the graphical model $\mathcal{G}\sim\CM_n$.
	\subsection{Warning Propagation in Graph}
	Consider a uniformly chosen vertex \(v\in V(\mathcal{G})\) in the graph \(\mathcal{G}(V,E)\), and let \(u_1, \ldots, u_{\partial_v}\in V(\mathcal{G})\) be its neighbors. 
	Then the rooted graph \((\mathcal{G}, v)\) converges locally in probability to the previously defined rooted unimodular Galton--Watson tree \((\mathcal{T}, o)\).
	Since the message \(\mu_{u_i \to v}^{(t)}\) depends only on the subgraph induced by the vertices at the graph distance at most \(t+1\) from \(v\), and the WP-induced Node Labeling propagation \(\mu_v^{(t)}\) depends only on the incoming messages \(\{\mu_{u_i \to v}^{(t)}\}_{i \in [\partial_v]}\), it follows that \(\mu_v^{(t)}\) is a function of the \((t+1)\)-neighborhood of \(v\). 
	Consequently,
	\begin{align}\label{eq:conv-p=1}
		\lim_{n\to\infty}\frac{1}{n}\sum_{v\in[n]}\ind{\mu_{v}^{(t)}=\msgU}\xrightarrow{p} \Pr\bc{\tilde\mu_{o}^{(t)}=\msgU}=\phi\bc{q_{\msgM}^{(t)}+q_{\msgU}^{(t)}}-\phi\bc{q_{\msgM}^{(t)}}-q_{\msgU}^{(t)}\phi'\bc{q_{\msgM}^{(t)}}.
	\end{align}
	On the other hand, by \Cref{thm_WP_Joon}, for all \(\varepsilon, \delta > 0\), there exists \(t_0 = t_0(\delta)\) such that, with high probability, for all \(t \geq t_0\), we have
	\[
	\Pr\bc{ \sum_{(u,v) \in E(\mathcal{G})} \mathbf{1}\left\{ \mu_{u \to v}^{(t)} \neq \mu_{u \to v}^{(t_0)} \right\} > \delta n } < \varepsilon,
	\]
	where convergence is measured with respect to the distance $ \mathscr{D}(\cdot,\cdot) $. 
	If for all neighbors $u\in V(\mathcal{G})$ of $v$, $\mu_{u \to v}^{(t)} = \mu_{u \to v}^{(t_0)}$, then $\mu_{v}^{(t)} =\mu_{ v}^{(t_0)}$. 
	Hence, for all $t\geq t_0$,
	\begin{align}\label{eq:conv-p-label}
		\Pr\bc{\sum_{v \in [n]} \mathbf{1}\left\{ \mu_{v}^{(t)} \neq \mu_{ v}^{(t_0)} \right\} > \delta n } < \varepsilon.
	\end{align}
	Note from \Cref{lem_WP_limit} that $\lim_{t\to\infty} \Delta_t=\Delta$. 
	For sufficiently large $t_0$,
	\begin{align}\label{eq:conv-t_0}
		\abs{\phi\bc{q_{\msgM}^{(t_0)}+q_{\msgU}^{(t_0)}}-\phi\bc{q_{\msgM}^{(t_0)}}-q_{\msgU}^{(t_0)}\phi'\bc{q_{\msgM}^{(t_0)}}-\phi(\alpha^*)+\phi(\alpha_*)+(\alpha^*-\alpha_*)\phi'(\alpha_*)}\leq \delta.
	\end{align}
	Combining \eqref{eq:conv-p=1}, \eqref{eq:conv-p-label}, and \eqref{eq:conv-t_0} yields
	\begin{align}\label{eq:limit-sum-ind}
		&\limsup_{n\to\infty}\Pr\bc{\limsup_{t\to\infty}\abs{\frac{1}{n}\sum_{v\in[n]}\ind{\mu_{v}^{(t)}=\msgU}- \phi(\alpha^*)+\phi(\alpha_*)+(\alpha^*-\alpha_*)\phi'(\alpha_*)} > 3\delta  }\\
		\leq&\limsup_{n\to\infty}\Pr\bc{\limsup_{t\to\infty}\abs{\frac{1}{n}\sum_{v\in[n]}\ind{\mu_{v}^{(t)}=\msgU}- \frac{1}{n}\sum_{v\in[n]}\ind{\mu_{v}^{(t_0)}=1}} > \delta  }\\
		&+\limsup_{n\to\infty}\Pr\bc{\abs{\frac{1}{n}\sum_{v\in[n]}\ind{\mu_{v}^{(t_0)}=\msgU}- \phi\bc{q_{\msgM}^{(t_0)}+q_{\msgU}^{(t_0)}}-\phi\bc{q_{\msgM}^{(t_0)}}-q_{\msgU}^{(t_0)}\phi'\bc{q_{\msgM}^{(t_0)}}} > \delta  } \leq \varepsilon.
	\end{align}
	\vspace{2mm}
	\begin{proof}[Proof of \Thm~\ref{thm_KSRT}]
		Recall from \Cref{re-wp-ks}
		that \(\mu_{v}^{(2t)} = \msgU\) if and only if vertex \(v\in V(\mathcal{G})\) will be in $E$ after \(t^{\text{th}}\) round of step (\ref{it_2}) in Algorithm \ref{alg_karpsipser}.  Hence,
		\begin{align}\label{eq:limit-ks}
			\lim_{t \to \infty} \frac{1}{n} \sum_{v \in [n]} \ind{\mu_v^{(2t)} = \msgU} = \frac{|\ks|}{n}.
		\end{align}
		Combining \eqref{eq:limit-sum-ind} and \eqref{eq:limit-ks}, and using the arbitrariness of \(\delta\) and \(\varepsilon\), the result follows.    
	\end{proof}
	\vspace*{5mm}
	\subsection*{\large{Acknowledgment}}
	We would like to thank Amin Coja-Oghlan and James Martin for valuable discussion.
	Joon Hyung Lee acknowledges support from the Dutch Research Council (NWO) as part of the project \emph{Boosting the Search for New Quantum Algorithms with AI (BoostQA)}, file number NGF.1623.23.033, under the research programme \emph{Quantum Technologie 2023}.
	Haodong Zhu is supported by the NWO Gravitation project NETWORKS under grant no. 024.002.003 and the European Union’s Horizon 2020 research
	and innovation programme under the Marie Skłodowska-Curie grant agreement no. 945045.

	\section{appendix}
	\subsection{Proof of Fact \ref{lem_WP_convergence}}
	Let us define a hierarchical order on the message sets with alphabet $\cbc{\msgL, \msgM, \msgU}$ as in \Rem~\ref{lem_hierarchy}.
	For the proof, we proceed in two parts by interpreting the inequality $\mu^{(t+1)}_{u \to v} \leq \mu^{(t)}_{u \to v}$ concerning their order:
	\begin{itemize}
		\item \emph{Monotonicity Property:} For the first part of the \Lem~\ref{lem_WP_convergence}, we prove by induction on $t$ that for all $(u,v)\in E$, the message sequence is non-increasing in this order:
		\[
		\mu^{(t+1)}_{u \to v} \leq \mu^{(t)}_{u \to v},
		\]		
		Base Case: At $t=0$, we have $\msg{u}{v}{0}=\msgU$ for all $(u,v)\in E$ as per initialization.
		Thus for the message at time $t=1$, $\msg{u}{v}{1}\in \{\msgL,\msgM,\msgU\}\leq \msgU$.
		More specifically, at time $t=1$, the message $\msg{u}{v}{1}$ can be either stuck to its previous iteration value i.e. $\msgU$ or can decrease to $\msgM$. 
		So, according to \Rem~\ref{lem_hierarchy} the base case holds true.\\
		Inductive Step: Assume that for some $t\geq 0$, we have $\msg{u}{v}{t}\leq \msg{u}{v}{t-1}$ for all $(u,v)\in E$.
		We must show that the same holds for time $(t+1)$.
		
		According to the WP update rule defined in \eqref{WP_update_rule} if a message $\msg{w}{u}{t} (\forall w\in \partial u\backslash\cbc{v})$ decreases from either $\msgU \to \msgM$ or from $\msgM \to \msgL$,then according to the Karp-Sipser leaf removal algorithm, the condition for returning to $\msgL$ or $\msgM$ is met sooner than before, that is the messages can decrease to $\msgM$ and $\msgL$ respectively in the consequent iterations.
		In particular, if $\msg{w}{u}{t}=\msgU$, then it only changes to $\msgM$ or $\msgL$, never to a larger value in the order.
		That is, at time $(t+1)$, the message $\msg{u}{v}{t+1}$ can be either any of the $\msgU$ or $\msgM$ messages and in the next iteration, it can be either of the $\msgU,\msgM$ or $\msgL$ messages.
		Hence, for all $(u,v)\in E$ the message update sequence is monotone non-increasing in its input under the message ordering scheme.
		Thus,
		\[
		\mu^{(t+1)}_{u \to v} \leq \mu^{(t)}_{u \to v},
		\]	
		completing the induction.
		\vspace*{2mm}
		\item \emph{Convergence:} Although the Warning Propagation behaves similarly in both directed and undirected graphs, in the case of a finite undirected graph $\mathcal{G}(V,E)\sim\CM_{n}$, internally it maintains the directed message from one node to another node.
		Thus, the asymmetry in the WP update rule \ref{WP_update_rule} introduces a direction even if the underlying graph $\mathcal{G}$ is undirected.
		We now show that the message sequence $\msg{u}{v}{t}$ stabilizes after at most $\abs{E}$ iterations.
		Note that for the edge $(u,v)\in E$ the message sequence $\msg{u}{v}{t} \in \cbc{\msgL,\msgM,\msgU}$ is non-increasing and can change at most twice:
		\[
		\msgU\to\msgM\to\msgL
		\]
		Hence, each message changes value atmost two times.
		Since the number of directed edges is $2\abs{E}$, the total number of updates over all messages is at most $4\abs{E}$.
		Now define $t^*$ to be the first time such that 
		\begin{align*}
			\msg{u}{v}{t^*+1} =\msg{u}{v}{t^*} ~~~\mbox{for all} ~(u,v) \in E(\mathcal{G})
		\end{align*}
		We claim that $t^*\leq \abs{E}$.
		
		Consider a dependency graph over messages: if at time $t$ the message $\msg{w}{u}{t}$ changes, then at time $(t+1)$ it may cause the message $\msg{u}{v}{t+1}$ to change for each $v \in \partial u \backslash\cbc{w}$ and this defines a dependency DAG (direct acyclic graph) between directed edges as well as messages.
		
		Due to monotonicity, the DAG is also acyclic and each edge depends on at most $\Delta(u)$ incoming messages (where $\Delta$ is the maximum degree of node $u$) for finite number of edges, the total number of rounds needed before all the messages stabilize is bounded by the length of the longest casual chain in the DAG -- at most $\abs{E}$ edges.   
	\end{itemize}
\end{document}